\documentclass[12pt]{article}
\usepackage{ct}

\usepackage{
algpseudocode, 
algorithm, 
tabularx, 
multicol, 
mathtools, 
ytableau, 
enumerate, 
tikz 
}
\usepackage[vcentermath,enableskew]{youngtab}

\newtheorem{thmIntro}{Theorem}   

\newtheorem{propIntro}[thmIntro]{Proposition}

\specs{2}{3 (2)}{2023}{}

\dateline{Mar 1, 2022}{Jan 19, 2023}{Sep, 15 2023}



\title{RSK tableaux and box-ball systems}

\author[1]{Ben Drucker\thanks{Supported by the University of Connecticut Mathematics REU and NSF (DMS-1950543).}}
\author[2]{Eli Garcia\thanks{Supported by the University of Connecticut Mathematics REU and NSF (DMS-1950543).}}
\author[3]{Emily Gunawan\thanks{Supported by the University of Connecticut Mathematics REU and NSF (DMS-1950543) and by the Isaac Newton Institute for Mathematical Sciences (funded by EPSRC Grant Number EP/R014604/1) during the programme \emph{Cluster algebras and representation theory}.}}
\author[4]{Aubrey Rumbolt\thanks{Supported by the University of Connecticut Mathematics REU and NSF (DMS-1950543).}}
\author[5]{Rose Silver\thanks{Supported by the University of Connecticut Mathematics REU and NSF (DMS-1950543).}}

\affil[1]{%
Pacific Northwest National Laboratory, Richland, WA, U.S.A

\email{ben.drucker@icloud.com}%
}

\affil[2]{%
Massachusetts Institute of Technology, Cambridge, MA, U.S.A

\email{tbone@mit.edu}%
}

\affil[3]{%
David and Judi Proctor Department of Mathematics, University of Oklahoma, Norman, OK, U.S.A

\email{egunawan@ou.edu}%
}

\affil[4]{%
Charles H. McCann Technical School, North Adams, MA, U.S.A

\email{ar366599@wne.edu}%
}

\affil[5]{%
Khoury College of Computer Sciences, Northeastern University, Boston, MA, U.S.A

\email{silver.r@northeastern.edu}%
}

\renewcommand \pi {v} 

\makeatletter
\newcommand*{\algrule}[1][\algorithmicindent]{\makebox[#1][l]{\hspace*{.5em}\vrule height .5\baselineskip depth .25\baselineskip}}%

\newcount\ALG@printindent@tempcnta
\def\ALG@printindent{%
    \ifnum \theALG@nested>0
        \ifx\ALG@text\ALG@x@notext
            \addvspace{-3pt}
        \else
            \unskip
            \ALG@printindent@tempcnta=1
            \loop
                \algrule[\csname ALG@ind@\the\ALG@printindent@tempcnta\endcsname]%
                \advance \ALG@printindent@tempcnta 1
            \ifnum \ALG@printindent@tempcnta<\numexpr\theALG@nested+1\relax
            \repeat
        \fi
    \fi
    }%
\usepackage{etoolbox} 
\patchcmd{\ALG@doentity}{\noindent\hskip\ALG@tlm}{\ALG@printindent}{}{\errmessage{failed to patch}} 
\makeatother

\algnewcommand{\LineComment}[1]{ 
    \State \hspace{0.5 em} \(\blacktriangleright\) #1}
 
\algblockdefx[customBlk]{Begin}{End}    
    {\textbf{begin} }
    {\textbf{end} }
    \algnewcommand{\Line,Comment}[1]{\State \hspace{0.5 em} \(\blacktriangleright\) #1}
    \newlength{\whilewidth}
    \algrenewcommand{\algorithmiccomment}[1]{\hspace{0.33em}$\blacktriangleright$ #1}

\renewcommand{\th}{$^\text{th}\text{ }$}
\DeclareMathOperator{\sh}{sh}
\DeclareMathOperator{\incr}{i}
\DeclareMathOperator{\decr}{d}
\DeclareMathOperator{\localincr}{I}
\DeclareMathOperator{\localdecr}{D}
\DeclareMathOperator{\SDself}{SD}
\DeclareMathOperator{\Ptself}{P}
\DeclareMathOperator{\Qtself}{Q}
\newcommand{\SD}[1]{\SDself(#1)}
\renewcommand{\P}[1]{\Ptself(#1)}
\newcommand{\Q}[1]{\Qtself(#1)}
\newcommand{\Qmaxtime}{\widehat{Q}}
\newcommand{\SnQmaxtime}{S_n (\widehat{Q})}
\newcommand{\lambdaRS}{\lambda}
\newcommand{\muRS}{\mu}
\newcommand{\lambdaBBS}{\Lambda}
\newcommand{\muBBS}{M}

\begin{document}

\maketitle


\begin{abstract}
A box-ball system is a discrete dynamical system whose dynamics come from the balls jumping according to certain rules. A permutation on $n$ objects gives a box-ball system state by assigning its one-line notation to $n$ consecutive boxes. After a finite number of steps, a box-ball system will reach a steady state. From any steady state, we can construct a tableau called the soliton decomposition of the box-ball system. We prove that if the soliton decomposition of a permutation $w$ is a standard tableau or if its shape coincides with the Robinson--Schensted (RS) partition of $w$, then the soliton decomposition of $w$ and the RS insertion tableau of $w$ are equal. We also use row reading words, Knuth moves, RS recording tableaux, and a localized version of Greene's theorem (proven recently by Lewis, Lyu, Pylyavskyy, and Sen) to study various properties of a box-ball system.
\end{abstract}



\section{Introduction}

A \emph{box-ball system (BBS)} is a collection of discrete time states. 
At each state, we have an injective map from $n$ balls (labeled by the integers from $1$ to $n$) to boxes (labeled by the natural numbers); 
each box can fit at most one ball. 
The dynamics come from the balls jumping according to certain rules. 
Let $S_n$ denote the set of permutations on $[n]=\{ 1, 2, \dots, n\}$. 
A permutation $w$ in $S_n$ gives a box-ball system state by assigning the one-line notation of the permutation to $n$ consecutive boxes. 
Given a BBS state at time $t$, we compute the BBS state at time $t+1$ by applying one BBS move, which is the process of moving 
 each integer to the nearest empty box to its right, beginning with the smallest. 
See Figure~\ref{fig:intro:452361:t=0 to t=1}. 
This version of the box-ball system was introduced in~\cite{Takahashi93} and is an extension of the box-ball system first invented by Takahashi and Satsuma in~\cite{TS90}.
\begin{figure}[htb]
\centering
\includegraphics[width = 0.55\textwidth]{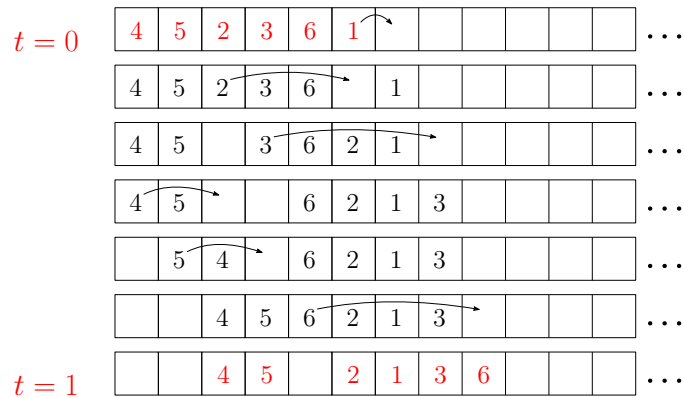}
\caption{Performing a BBS move on $w=452361$.}
\label{fig:intro:452361:t=0 to t=1}
\end{figure}

A \emph{soliton} is a maximal consecutive increasing sequence of balls which is preserved by all subsequent BBS moves. 
After a finite number of BBS moves, a box-ball system containing a configuration $w$ will reach a \emph{steady state}, 
decomposing into solitons
whose sizes are weakly decreasing from right to left, 
that is, forming an integer partition shape. 
From such a state, 
we can construct 
the \emph{soliton decomposition of the box-ball system}, denoted SD, 
by stacking solitons so that the rightmost soliton is placed on the first row, the soliton to its left is placed on the second row, and so on.
We obtain a tableau where each row is increasing but which may or may not be standard. The \emph{soliton decomposition of a permutation} $w$ is the soliton decomposition of the box-ball system containing $w$.

\begin{figure}[htb]
\centering
\includegraphics[width = 0.7\textwidth]{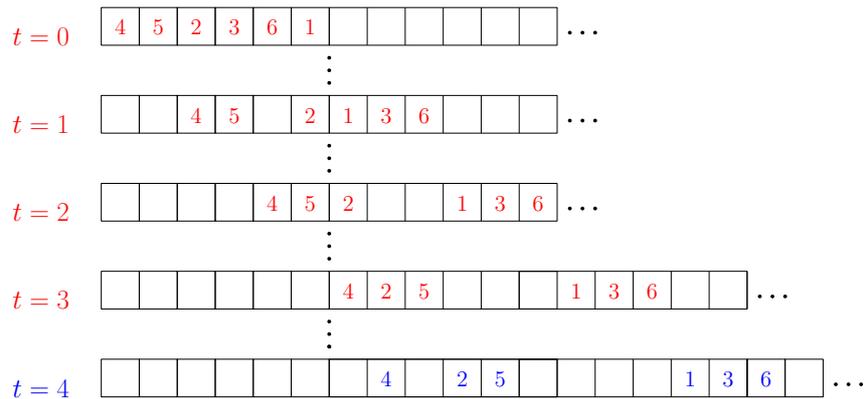}
 \caption{BBS moves starting at $w=452361$.} 
\label{fig:intro:452361:t=0 to t=4}
\end{figure}

Figure~\ref{fig:intro:452361:t=0 to t=4} shows the state of the box-ball system containing $w=452361$ from $t=0$ to~$t=4$. Note that steady state is first reached at $t=3$. 
The soliton decomposition of $w=452361$ is 
the tableau 
\[
{\small \SD{w}=
  ~
\young(136,25,4)}. 
\]
In this example, the soliton decomposition is a standard tableau, but most permutations have soliton decompositions which are not standard. 
The tableau $\SD{w}$ has shape $(3,2,1)$. 
We will refer to the shape of the soliton decomposition as the \emph{BBS soliton partition}.

The well-known Robinson--Schensted (RS) insertion algorithm is 
a bijection 
\[{w \mapsto (\P{w},\Q{w})}\]
from $S_n$ onto pairs of standard size-$n$ tableaux of the same shape~\cite{Sch61}. 
The tableau $\P{w}$ is called the \emph{insertion tableau} of $w$, and the tableau $\Q{w}$ is called the \emph{recording tableau} of $w$. The shape of these tableaux is called the \emph{RS partition of $w$}.

The \emph{row reading word} of a 
 tableau is the permutation formed by concatenating the rows of the tableau from bottom to top, left to right. 
\begin{equation}
\label{eq:if r is reading word of T then P(r)=T}
\text{If $r$ is the row reading word of a standard tableau $T$, then $\P{r}=T$.} 
\end{equation}
For example, 
if $w=452361$, then 
\[
\P{w}=\young(136,25,4), 
~~
\Q{w}=\young(125,34,6).
\]
The tableau {$\P{w}$} 
has row reading word $r=425136$. 
The insertion tableau of $r$ is the ta\-bleau~$\P{w}$. 
For more information, see for example the textbook~\cite[Section~7.5]{Sag20}.

The \emph{carrier algorithm} given in~\cite{fukuda04} (which we review in Section~\ref{sec:carrier}) is a way to transform a box-ball configuration at time $t$ into the configuration at time $t+1$. At each step in the algorithm, we insert and bump numbers in and out of a \emph{carrier} filled with a weakly increasing sequence, following a rule which should remind the reader of the RS insertion algorithm. 

Our goal is to study the connection between the soliton decompositions
and RS tableaux of permutations. 
We now describe our main results.

\subsection{Insertion tableaux and soliton decompositions}

For the permutation $w=452361$ used in the above example, we have $\SD{w}=\P{w}$. 
However, in general the soliton decomposition and the RS insertion tableau of a permutation do \emph{not} coincide. 
Surprisingly, having a standard soliton decomposition tableau  
or 
having a BBS soliton partition 
which equals the RS partition 
is enough to guarantee that the soliton decomposition and the RS insertion tableau coincide. 

\begin{thmIntro}[Theorem~\ref{thm:tfae}]\label{thmintro:tfae}
Suppose $w$ is a permutation. 
Then 
the following are equivalent:
\begin{enumerate}
\item 
$\SD{w} = \P{w}.$
\item 
$\SD{w}$ is a standard tableau.
\item 
The shape of $\SD{w}$ equals the shape of $\P{w}$.
\end{enumerate}
\end{thmIntro} 
The key ingredients of our proof are Greene's theorem (Theorem~\ref{thm:Greene's theorem}) and a result of Fukuda which says that the RS insertion tableau is an invariant of a box-ball system (Theorem~\ref{thm:preserved}).
The proof that part~(\ref{tfae:itm:three}) implies part~(\ref{tfae:itm:two}) was suggested to us by Darij Grinberg. 

\subsection{Tableau reading words}
We study the connection between steady-state configurations and row reading words. 

\begin{propIntro}[Proposition~\ref{prop:t=0}] 
\label{propintro: forwards t = 0}
A permutation $r$ is in steady state if and only if $r$ is the row reading word of a standard tableau. 
\end{propIntro}
 
Next, we represent a box-ball system state as an array containing integers from $1$ to $n$ called the \emph{configuration array}. This array has increasing rows but not necessarily increasing columns; it also may not have a valid skew shape and it may be disconnected. 
Proposition~\ref{propintro: forwards t = 0} turns out to be a special case of the following. 

\begin{propIntro}
[Proposition~\ref{prop:t=0 generalization}]
\label{propintro: t=0 generalization} 
A BBS configuration $w$ is in steady state if and only if the configuration array of $w$ is 
a standard skew tableau whose rows are weakly decreasing in length. 
\end{propIntro}

We will prove Proposition~\ref{propintro: t=0 generalization} in Section~\ref{sec:reading words and steady states} using the carrier algorithm. Note that Proposition~\ref{propintro: t=0 generalization} is a corollary of a characterization for steady state given by Lewis, Lyu, Pylyavskyy, and Sen in~\cite[Proof of Lemma 2.1 and 2.3]{LLPS19}.

\subsection{Recording tableaux and time to steady state}

We also study the relationship between the RS recording tableau of a permutation and the behavior of its box-ball system. 
The number of BBS moves required for a permutation $w$ to reach 
 steady state is called the \emph{steady-state time} of $w$.
For example, 
as illustrated in Figure~\ref{fig:intro:452361:t=0 to t=4}, 
the steady-state time of the permutation $452361$ is $3$.

\begin{thmIntro}[Theorem~\ref{thm:n minus 3}]\label{thmintro:n minus 3} 
If $n \geq 5$, 
let 
\[\Qmaxtime
\coloneqq
\begin{ytableau}
 1 & 2 &\none[\hdots]&\scalebox{.83}{$n$-$2$}&\scalebox{.83}{$n$-$1$}\\
 3 & 4\\
 n
 \end{ytableau}.
 \]
If $\Q{w}=\Qmaxtime$, 
 then $w$ first reaches steady state at time $n-3$. 
\end{thmIntro}

This particular recording tableau is special; 
we conjecture that all other permutations in $S_n$ have  steady-state time smaller than $n-3$. 

\begin{conjecture}
\label{conj: n minus 3}
A permutation in $S_n$ 
whose recording tableau is not equal to $\Qmaxtime$
has steady-state time smaller than $n-3$.
\end{conjecture}

Furthermore, we conjecture that Theorem~\ref{thmintro:n minus 3} is a special case of the following general phenomenon. 

\begin{conjecture}
\label{conj: same Q implies same t}
If two permutations $\pi$ and $w$ are such that $\Q{\pi}=\Q{w}$, then $\pi$ and $w$ have the same steady-state time.
\end{conjecture}
Conjecture~\ref{conj: same Q implies same t} is proven in a sequel to this paper~\cite{sumry21}. 

\begin{remark}
Conjecture~\ref{conj: same Q implies same t} would simplify the proof of Theorem~\ref{thmintro:n minus 3}; it would simply require demonstrating that one single permutation whose recording tableau is $\Qmaxtime$ has steady-state \linebreak time~${n-3}$. 
\end{remark}

\subsection{Types of Knuth moves}

The RS insertion tableau is preserved under any Knuth move~\cite{Knuth70}. 
In contrast, the soliton decomposition is only preserved under certain types of Knuth moves. 

\begin{definition}[Knuth Moves]\label{def: Knuth moves}
Suppose $\pi$, $w \in S_n$ and $x<y<z$.

\begin{enumerate}
\item We say that 
$\pi$ and $w$ differ by a Knuth relation of the \emph{first kind} ($K_1$) if 
\[
\pi=\pi_1 \dots \textcolor{black}{y}\textcolor{black}{x}\textcolor{black}{z} 
\dots \pi_n
~ \text{and} ~
w=\pi_1 \dots \textcolor{black}{y}\textcolor{black}{z}\textcolor{black}{x} \dots \pi_n
\text{ or 
vice versa.}
\]
        
\item We say that 
$\pi$ and $w$ differ by a Knuth relation of the \emph{second kind} ($K_2$) if 
\[
\pi=\pi_1 \dots \textcolor{black}{x}\textcolor{black}{z}\textcolor{black}{y} \dots \pi_n
~
\text{and}
~
w=\pi_1 \dots \textcolor{black}{z}\textcolor{black}{x}\textcolor{black}{y} \dots \pi_n
\text{ or 
vice versa.}
\]
\end{enumerate}
In addition, 
We say that 
$\pi$ and $w$ differ by a Knuth relation of \emph{both kinds} ($K_B$) if 
they differ by a Knuth relation of the first kind ($K_1$) and of the second kind ($K_2$), that is, 
\[
\pi=\pi_1 \dots \textcolor{black}{y_1}\textcolor{black}{x}\textcolor{black}{z}\textcolor{black}{y_2} \dots \pi_n ~ \text{and} ~ w=\pi_1 \dots \textcolor{black}{y_1}\textcolor{black}{z}\textcolor{black}{x}\textcolor{black}{y_2} \dots \pi_n
\text{ or vice versa}
\]
where 
 $x<y_1<z$ and $x<y_2<z$.
 
Note that, when we apply a $K_1$ move (respectively, a $K_2$ move), the move may or may not be a $K_B$ move. 
If we apply a $K_B$ move, then it is both a $K_1$ move and a $K_2$ move.

A \emph{proper $K_1$ move} is a $K_1$ move which is not $K_B$, and a \emph{proper $K_2$ move} is a $K_2$ move which is not $K_B$.

When performing a Knuth move, if we replace an ``$xz$'' pattern with a ``$zx$'' pattern, we denote this with a superscript ``$+$.'' 
Otherwise, if we replace a ``$zx$'' pattern with an ``$xz$'' pattern, we denote this with a superscript ``$-$.'' For example, 
if $x< y_1 <z$ and $x< y_2 <z$, 
the move~$y_1 x z y_2\mapsto y_1 z x y_2$ is denoted $K_B^+.$ 

We say that $\pi$ and $w$ are \emph{Knuth equivalent} if 
they differ by a finite sequence of Knuth relations. 
\end{definition}

Using the localized version of Greene's Theorem 
given in Section~\ref{sec:local Greene's theorem}, 
we prove a partial characterization of 
the BBS soliton partition
in terms of types of Knuth moves.

\begin{thmIntro}
[Theorem~\ref{thm:knuth paths}]
\label{thmintro: knuth paths}
If $\pi$ and $w$ are related by a sequence of Knuth moves containing an odd number of $K_B$ moves, then $ \SD{\pi}
\neq
\SD{w}$. 
If $\pi$ and $w$ are related by a sequence of non-$K_B$ Knuth moves, then $\sh\SD{\pi}=
\sh\SD{w}$. 
\end{thmIntro} 

We also use non-$K_B$ Knuth moves to give a family of permutations which have steady-state time~$1$.

\begin{thmIntro}[Theorem~\ref{thm:t=1}]
\label{thmintro: t=1}
Let $r$ be the row reading word of a standard tableau. 
If $w$ is a permutation 
which is related to $r$ by 
one proper $K_1$ move or one proper $K_2$ move, 
then the steady-state time of $w$ is~$1$.
\end{thmIntro}

The paper is organized as follows. 
In the next two sections, we review materials in the literature that we will use to prove our results. 
First, we review Greene's theorem in Section~\ref{sec:Greene's theorem}
 and Lewis, Lyu, Pylyavskyy, and Sen's localized Greene's theorem in Section~\ref{sec:local Greene's theorem}. 
Next, we review Fukuda's carrier algorithm and its connection to the RS insertion tableaux in Section~\ref{sec:carrier}. 
In Section~\ref{sec:proof of tfae}, we prove Theorem~\ref{thmintro:tfae}. 
In Section~\ref{sec:reading words and steady states}, 
we define the configuration array and 
 use the carrier algorithm to prove 
Proposition~\ref{propintro: t=0 generalization}. 
Section~\ref{sec5} is devoted to the proof of Theorem~\ref{thmintro:n minus 3}. 
We prove the two results involving types of Knuth moves 
(Theorem~\ref{thmintro: knuth paths} and Theorem~\ref{thmintro: t=1}) 
in Section~\ref{sec:knuth}.

\section{Greene's theorem and a localized version of Greene's theorem}
\label{sec:Greene's theorem sec:local Greene's theorem}

In the 1970s, Greene showed that the RS partition of a permutation and its conjugate record the numbers of disjoint unions of increasing and decreasing sequences of the permutation, which we explain in Section~\ref{sec:Greene's theorem}. 
Lewis, Lyu, Pylyavskyy, and Sen recently showed that the BBS soliton partition of a permutation and its conjugate record a localized version of Greene's theorem statistics. 
They studied an alternate version of the box-ball system, so in Section~\ref{sec:local Greene's theorem} we reframe their result to match our box-ball convention.

\subsection{Greene's theorem and RS partition}
\label{sec:Greene's theorem}

In this section, we review 
Greene's theorem~\cite[Theorem~3.1]{Gre74}, which states that the RS partition of a permutation and its conjugate record the numbers of disjoint unions of increasing and decreasing sequences of the permutation. 
For more details, see for example Chapter 3 of the textbook~\cite{Sag01}. 

\begin{definition}[longest $k$-increasing and $k$-decreasing subsequences]
A subsequence $\sigma$ of $w$ is called \emph{k-increasing} if, as a set, it can be written as a disjoint union 
\[
\sigma = \sigma_1 \sqcup \sigma_2 \sqcup \cdots \sqcup \sigma_k
\]
where each $\sigma_i$ is an increasing subsequence of $w$. 
If each $\sigma_i$ is a decreasing subsequence of $w$, we say that $\sigma$ is \emph{k-decreasing}. 
Let 
\begin{align*}
& \incr_k(w) \text{ denote the length of a longest $k$-increasing subsequence of $w$} \\
\text{and} & \decr_k(w) \text{ denote the length of a longest $k$-decreasing subsequence of $w$.}
\end{align*}
\end{definition}

\begin{theorem}[{\cite[Theorem~3.1]{Gre74}}]
\label{thm:Greene's theorem}
Suppose $w \in S_n$. 
Let 
$\lambdaRS=(\lambdaRS_1, \lambdaRS_2, \lambdaRS_3, \dotsc)$ denote the RS partition of $w$, that is, let 
$\lambdaRS=\sh \P{w}$. 
Let $\muRS=(\muRS_1, \muRS_2, \muRS_3, \dotsc)$ denote the conjugate of $\lambdaRS$. 
Then, for any $k$,
\begin{align*}
\incr_k(w) &= \lambdaRS_1 +\lambdaRS_2 +\dotsc + \lambdaRS_k, 
\\
\decr_k(w) &= \muRS_1 +\muRS_2 +\dotsc + \muRS_k. 
\end{align*}
\end{theorem}

\begin{example}
\label{ex:thm:Greene's theorem:5623714}
Let $w=5623714$. 
For short, we write $\incr_k \coloneqq \incr_k (w)$
and $\decr_k \coloneqq \decr_k (w)$. 
The longest $1$-increasing subsequences are 
\[ 
5 6 7, 
\quad 
2 3 7,
\quad 
\text{ and } \quad 
2 3 4. 
\]
The longest $2$-increasing subsequence is given by
\[ 562374 = 567 \sqcup 234 .\]
A longest $3$-increasing subsequence (among others) 
is given by
\[
5623714 = 56 \sqcup 237 \sqcup 14.
\]
Thus,
\[
\incr_1 =3, \qquad
\incr_2 =6, \qquad \text{ and } \qquad
\incr_k =7 ~ \text{ if } ~ k\geq 3. 
\]

Similarly, the longest $1$-decreasing subsequences are 
\[
521, \quad 
621, \quad 
531, \quad \text{ and } \quad 
631. 
\]
A longest $2$-decreasing subsequence 
(among others) 
is given by
\[
52714 = 
521 
\sqcup 
74. 
\]
A longest $3$-decreasing subsequence 
(among others) 
is given by
\[
5623714 =
52
\sqcup 
631
\sqcup 
74.
\]
Thus, 
\[
\decr_1 =3, \qquad \decr_2 =5, \qquad
\text{ and } 
\qquad
\decr_k = 7 ~\text{ if }~ k \geq 3. 
\]

By Theorem~\ref{thm:Greene's theorem}, 
the RS partition is equal to 
$\lambdaRS=(\incr_1,\incr_2-\incr_1, \incr_3-\incr_2)=(3,3,1)$ and the conjugate of the RS partition is $\muRS=(\decr_1,\decr_2-\decr_1, \decr_3-\decr_2)=(3,2,2)$. 
We can verify this by computing the RS tableaux 
\[
\P{w}=\young(134,267,5), 
\qquad
\Q{w}=
\young(125,347,6). 
\]
\end{example}

\subsection{Localized Greene's theorem and BBS soliton partition}
\label{sec:local Greene's theorem}
 
In~\cite[Lemma 2.1]{LLPS19} and the blog post~\cite{Lewis20Blog},  Lewis, Lyu, Pylyavskyy, and Sen presented a localized version of  Greene's theorem. 
They studied an alternate version of the box-ball system, and in this section we reframe their result to match our box-ball convention. 

\begin{definition}[A localized version of longest $k$-increasing subsequences]
\label{defn: lambda and mu}
If $u$ is a sequence, 
let $\incr(u)$ denote the length of a longest increasing subsequence of $u$. 

For $w \in S_n$ and $k\geq 1,$ we define \[
\localincr_k(w)= \max_{w=u_1\mid\cdots\mid u_k}\sum_{j=1}^k \incr(u_j), \]
where the maximum is taken over ways of writing $w$ as a concatenation $u_1 \mid \dots \mid u_k$ of consecutive subsequences. 
That is, we consider all ways to break $w$ into $k$ consecutive subsequences, sum the $\incr(u_j)$ values for each way, and let $\localincr_k(w)$ be the maximum sum. 
\end{definition}

If $u$ is a sequence of $\ell$ elements, 
an integer $m\in [\ell-1]$ is called a \emph{descent} of $u$ if $u_m > u_{m+1}$. 

\begin{definition}[A localized version of longest $k$-decreasing subsequences]
Let \[\localdecr(u) 
\coloneqq 
\begin{cases}
0 &\text{if $u$ is empty}
\\
1+|\{$descents of $u\}| & \text{otherwise}.
\end{cases}\] 
For $w \in S_n$ and $k\geq 1,$ we define 
\[ \localdecr_k(w) = \max_{w= u_1 \sqcup \cdots \sqcup u_k} \sum_{j=1}^k \localdecr(u_j),\] 
where the maximum is taken over
ways to write $w$ as the union of 
disjoint subsequences $u_j$ of $w$. 
Notice that we only require $u_1,\dots, u_k$ to be disjoint, \emph{not} consecutive, in contrast to the procedure for calculating $\localincr_k(w)$.
\end{definition}

The following lemma is a corollary of {\cite[Lemma 2.1]{LLPS19}}.

\begin{lemma}[A localized version of Greene's theorem]
\label{lem:local Greene's theorem}
Suppose that $w \in S_n$.
Let  \linebreak 
$\lambdaBBS=(\lambdaBBS_1, \lambdaBBS_2, \lambdaBBS_3, \dotsc)$ denote the BBS soliton partition of $w$, that is, let 
$\lambdaBBS=\sh \SD{w}$. 
Let $\muBBS=(\muBBS_1, \muBBS_2, \muBBS_3, \dotsc)$ denote the conjugate of $\lambdaBBS$. 
Then, for any $k$,
\begin{align*}
\localincr_k(w) &= \lambdaBBS_1 +\lambdaBBS_2 +\dotsc + \lambdaBBS_k, 
\\
\localdecr_k(w) &= \muBBS_1 +\muBBS_2 +\dotsc + \muBBS_k.
\end{align*}
\end{lemma}

\begin{example}
\label{ex:lem:local Greene's theorem:5623714}
Let $w=5623714$, the permutation used in Example~\ref{ex:thm:Greene's theorem:5623714}. 
For short, we write $\localincr_k \coloneqq \localincr_k (w)$
and $\localdecr_k \coloneqq \localdecr_k (w)$. 
Then 
\begin{align*}
\localincr_1 &= \incr(w) = 3 \text{ (since the longest increasing subsequences are $5 6 7$, $2 3 7$, and $2 3 4$)},  \\
\localincr_2 &= 5 \text{ (witnessed by $56|23714$ or $56237|14$)}, \\
\localincr_3 &= 7 \text{ (witnessed uniquely by $56|237|14$), and } \\
\localincr_k &= 7 \text{ for all $k\geq 3$}.
\end{align*}

We have 
\begin{align*}
\localdecr_1 &= \localdecr (w) = 1 + |\text{descents of } 5623714| = 1 + |\{ 2, 5 \}| = 3, \\
\localdecr_2 &= 6 \text{ (one can take subsequences $531$ and $6274$, among other partitions),} \\
\localdecr_3 &= 7 \text{ (one can take subsequences 
$52$, $631$, and $74$, among other partitions), and}
\\
\localdecr_k &= 7 \text{ for all $k \geq 3$.}
\end{align*}

By Lemma~\ref{lem:local Greene's theorem}, we have that
$\sh \SD{w}=(\localincr_1,\localincr_2-\localincr_1, \localincr_3-\localincr_2)=(3,2,2)$ and its conjugate is $(\localdecr_1,$ $\localdecr_2-\localdecr_1$, $\localdecr_3-\localdecr_2)=(3,3,1)$. 
We can verify this by computing the soliton decomposition~$\SD{w}$, which turns out to be the nonstandard tableau
\[\young(134,27,56).\]
Note that, in this example, $\SD{w} \neq \P{w}$, demonstrating Theorem~\ref{thmintro:tfae}. Also, in this example, 
$\sh \SD{w}=(3,2,2)$ is smaller than 
$\sh \P{w}=(3,3,1)$ in the dominance partial order.
\end{example}

 \begin{corollary}
 If $w \in S_n$, then 
 the BBS soliton partition of $w$
 is smaller or equal to the RS partition of $w$ in the dominance partial order. 
\end{corollary}
\begin{proof}
 Let $\lambdaBBS = (\lambdaBBS_1, \lambdaBBS_2, \lambdaBBS_3, \dots)$ denote $\sh \SD{w}$ and 
 let $\lambdaRS = (\lambdaRS_1, \lambdaRS_2, \lambdaRS_3, \dots)$ denote
 $\sh \P{w}$. 
Then, for all $k=1,2,\dots$, we have 
\begin{align*}
\lambdaBBS_1 + \lambdaBBS_2 + \dots + \lambdaBBS_k &= \localincr_k (w) ~\text{ by localized Greene's theorem (Lemma~\ref{lem:local Greene's theorem})}\\
& \leq \incr_k (w)~ \text{ since $\localincr_k(w)$ gives the length of a $k$-increasing subsequence of $w$}\\
&= \lambdaRS_1 + \lambdaRS_2 + \dots + \lambdaRS_k ~\text{ by Greene's theorem (Theorem~\ref{thm:Greene's theorem})}.
\qedhere
\end{align*}
 \end{proof}

\section{Fukuda's carrier algorithm}
\label{sec:carrier}

In this section, 
we review the carrier algorithm and the fact that the RS insertion tableau is an invariant of a box-ball system (BBS). 

\subsection{Carrier algorithm}
The carrier algorithm is a way to describe a BBS move as a sequence of local operations of inserting and bumping numbers in and out of a carrier filled with a weakly increasing string. 
A version of the carrier algorithm was first introduced in~\cite{TM97carrier}, and
the version of the carrier algorithm we use in this paper comes from~\cite[Section~3.3]{fukuda04}. 
Given a BBS state at time~$t$, 
the carrier algorithm is used to calculate the state at time $t+1$. 
We describe the process in Algorithm~\ref{alg:carrier}. 
Note that, after each insertion and ejection step, the sequence in the carrier is weakly increasing.

\begin{algorithm}[p]
\begin{algorithmic}[1]
\Begin carrier algorithm
\State Set $e\coloneqq n+1$, so that $e$ is considered to be larger than any ball 
\State\parbox[t]{\dimexpr\linewidth-\algorithmicindent}{Set $B\coloneqq$ the configuration of the BBS at time $t$, 
where each empty box is replaced with an $e$ and the first (leftmost) element of $B$ is the integer in the first (leftmost) nonempty box in the configuration and the last (rightmost) element of $B$ is the integer in the last (rightmost) nonempty box of the configuration}
\State Let $\ell$ denote the number of elements (including the $e$'s) of $B$ 
\State Fill the ``carrier'' $\mathcal{C}$ ---depicted \raisebox{.5em}{$\underbracket{\phantom{ \dots \dots }}$}---with $n$ copies of $e$
\State Write $B$ to the right of $\mathcal{C}$
\Begin Process 1: insertion process
\ForAll {$i$ in $\{1,2,\dots,\ell\}$}
\State Set $p$ to be the $i$\th leftmost element of $B$
\Begin element ejection process 
\If{an element in $\mathcal{C}$ is larger than $p$} 
\State Set $s\coloneqq$ smallest element in $\mathcal{C}$ larger than $p$. If $s=e$, pick the leftmost $e$
\State Eject $s$ from $\mathcal{C}$ and put it immediately to the left of $\mathcal{C}$ 
\State insert $p$ in the place of $s$ 
\Else
\State Set $s\coloneqq$ the leftmost (smallest) element in $\mathcal{C}$ \label{alg:carrier:begin:element ejection process} 
\State Eject $s$ from $\mathcal{C}$ and put it immediately to the left of $\mathcal{C}$
\LineComment{Note: There are now $n-1$ elements in $\mathcal{C}$}
\State 
Shift each element of $\mathcal{C}$ to the left by one
\State 
Place $p$ in the rightmost location in $\mathcal{C}$

\LineComment{Note: There are now $n$ elements in $\mathcal{C}$}

\EndIf
\End element ejection process
\EndFor

\End Process 1: insertion process

\Begin Process 2: flushing process
\While{there are non-$e$ elements in $\mathcal{C}$}
\State Set $p\coloneqq e$ 
\State Perform the element ejection process (see line~\ref{alg:carrier:begin:element ejection process})
\EndWhile
\End Process 2: flushing process
\LineComment{\parbox[t]{\dimexpr\linewidth-\algorithmicindent-23pt}{
Note: The current elements to the left of $\mathcal{C}$ correspond to the $t+1$ state of the BBS}}
\End carrier algorithm 
\end{algorithmic}
\caption{The carrier algorithm~\cite{fukuda04}.}\label{alg:carrier}
\end{algorithm}

\begin{example}
\label{ex:carrier:452361:t=2 to t=3}
We compute the configuration at time $t=3$
of the box-ball system from Figure~\ref{fig:intro:452361:t=0 to t=4} by applying the carrier algorithm to the configuration at time $t=2$. 
Following Algorithm~\ref{alg:carrier}, we set $B\coloneqq 452ee136$. 
The carrier algorithm then proceeds as follows.
\begin{align*}
\text{\textbf{begin}} &\text{ Process 1:} \text{ insertion process}\\
&\underbracket{eeeeee}452ee136\\
e&\underbracket{4eeeee}52ee136\\
ee&\underbracket{45eeee}2ee136\\
ee4&\underbracket{25eeee}ee136\\
ee42&\underbracket{5eeeee}e136\\
ee425&\underbracket{eeeeee}136\\
ee425e&\underbracket{1eeeee}36\\
ee425ee&\underbracket{13eeee}6\\
ee425eee&\underbracket{136eee}\\
\text{\textbf{end}} &\text{ insertion process}
\\[1mm]
 \text{\textbf{begin}} \text{ Process 2:} &\text{ flushing process}\\
ee425eee&\underbracket{136eee}\leftarrow e\\
ee425eee1&\underbracket{36eeee}\leftarrow e\\
ee425eee13&\underbracket{6eeeee}\leftarrow e\\
ee425eee136&\underbracket{eeeeee}\\
\text{\textbf{end}} \text{ flushing} & \text{ process}
\end{align*}
The elements $ee425eee136$ to the left of $\mathcal{C}$ correspond to the configuration at time $t=3$ given in Figure~\ref{fig:intro:452361:t=0 to t=4}.
\end{example}

\subsection{The RS insertion tableau is an invariant of a box-ball system}
  
\begin{remark}[{\cite[Remark 4]{fukuda04}}]\label{rmk: carrier is knuth}
The carrier algorithm can be viewed as a sequence of Knuth moves. 
Consider the insertion of $\mathbf{p}$ into the carrier. 
Note that, since our carrier can carry $n$ elements, if $\mathbf{p} \neq e$, then the carrier must contain a number (possibly $e$) greater than $\mathbf{p}$. If $\mathbf{p}=e$, then no number in the carrier is greater than~$\mathbf{p}$. 

First, suppose $\mathbf{p} \neq e$, and let $C_p$ denote the smallest element in the carrier which is greater than~$\mathbf{p}$. 

\begin{enumerate}[(i)]
\item 
    If $C_p$ is the smallest element in the carrier, then the insertion process is equivalent to applying a sequence of $K_1^-$ moves \begin{center}
      $\underbracket{C_p z_1 z_2 \cdots z_{\ell-1}z_\ell} \, \mathbf{p}$ \\[2mm]
      ${C_p z_1 z_2\cdots z_{\ell-1} \mathbf{p} z_\ell}$ \\
      $\vdots$ \\
      ${C_p z_1 \mathbf{p} z_2\cdots z_{\ell - 1} z_\ell}$\\
      $C_p\underbracket{\mathbf{p} z_1 z_2 \dots \dots z_\ell}.$\\[2mm]
    \end{center} 
\item 
    If $C_p$ is the largest element in the carrier, then the insertion process is equivalent to applying a sequence of $K_2^+$ moves 
    \begin{center}
     $\underbracket{x_1 x_2\cdots x_{m-1}x_mC_p} \, \mathbf{p}$ \\[2mm]
      ${x_1 x_2\cdots x_{m-1}C_p x_m \mathbf{p}}$ \\
      $\vdots$ \\
      $x_1 C_p x_2\cdots x_{m-1}x_m \mathbf{p}$
      \\
      $C_p\underbracket{x_1 x_2\cdots x_{m-1}x_m \mathbf{p}}.$\\[2mm]
\end{center} 
\item 
    If $C_p$ is neither the smallest nor the largest element in the carrier, 
    then the insertion process is equivalent to applying a sequence of $K_1^-$ moves
\begin{center}
  $\underbracket{x_1 x_2 \cdots x_{m-1} x_m C_p z_1 z_2 \cdots z_{\ell-1}z_\ell} \, \mathbf{p}$ \\[2mm]
      ${x_1 x_2 \cdots x_{m-1} x_m C_p z_1 z_2 \cdots z_{\ell-1} \mathbf{p} z_\ell}$ \\
      $\vdots$ \\
      ${x_1 x_2 \cdots x_{m-1} x_m C_p z_1 \mathbf{p} z_2 \cdots z_{\ell-1}z_\ell}$\\
      ${x_1 x_2 \cdots x_{m-1} x_m C_p \mathbf{p} z_1 z_2 \cdots z_{\ell-1}z_\ell}$
    \end{center}
    followed by a sequence of $K_2^+$ moves 
    \begin{center}
     ${x_1 x_2 \cdots x_{m-1}x_m C_p \mathbf{p} z_1 z_2 \cdots z_{\ell-1} z_\ell}$ \\[2mm]
      ${x_1 x_2 \cdots x_{m-1}C_p x_m \mathbf{p} z_1 z_2 \cdots z_{\ell-1} z_\ell}$ \\
      $\vdots$ \\
      $x_1 C_p x_2 \cdots x_{m-1}x_m \mathbf{p} z_1 z_2 \cdots z_{\ell-1} z_\ell$\\
      $C_p\underbracket{x_1 x_2 \cdots x_{m-1}x_m \mathbf{p} z_1 z_2 \cdots z_{\ell-1} z_\ell}.$\\[2mm]
    \end{center}
\end{enumerate}    
Next, suppose $\mathbf{p}=e$. 
Then $\mathbf{p}$ is greater than or equal to every element in the carrier, and 
the insertion process is equivalent to applying the trivial transformation
\begin{align*} &\underbracket{x_1 x_2 \cdots x_n}\: \mathbf{p} \\
x_1\: &\underbracket{x_2 \cdots x_n ~ \mathbf{p}}.
\end{align*}
\end{remark}

\begin{theorem}
[{\cite[Theorem 3.1]{fukuda04}}]
\label{thm:preserved}
The RS insertion tableau is a conserved quantity under the time evolution of the BBS, i.e., the RS insertion tableau is preserved under each BBS move. 
More precisely, let $B_t$ be the state of a box-ball system at time $t$. 
Let $B_t'$ be the permutation created from $B_t$ by removing all $e$'s.
Then $\P{B_t'}$ is identical for all $t$.
\end{theorem}
  
\begin{example}
\label{ex:thm:preserved:452361}
As shown in Figure~\ref{fig:intro:452361:t=0 to t=4}, the configurations $452361$, $ee45e2136$, $eeee452ee136$, and $eeeeee425eee136$ are in the same box-ball system. 
As Theorem~\ref{thm:preserved} tells us, the permutations~$452361$, $452136$, and $425136$ have the same RS insertion tableau
\[
\P{452361}=\P{452136}=\P{425136} = 
\young(136,25,4).
\]
\end{example}  
  
\begin{corollary}  
\label{cor:P(w)=P(r)} 
Let $w$ be a permutation. 
If $r$ is the row reading word of $\SD{w}$, then ${\P{w} \! = \!\P{r}}$. 
\end{corollary}
\begin{proof}
Let $r$ be the row reading word of $\SD{w}$. 
By definition of the soliton decomposition tableau, 
we know that $r$ is the order in which the balls of $w$ are configured once we reach a steady state. Therefore, $r$ is a state in the box-ball system containing $w$. Theorem~\ref{thm:preserved} 
tells us that the RS insertion tableau is preserved under a sequence of box-ball moves, so
$\P{w}=\P{r}$.   
\end{proof}

\begin{example}
\label{ex:cor:P(w)=P(r):5623714}
Let $w=5623714$, the permutation from Section~\ref{sec:Greene's theorem sec:local Greene's theorem}, and let $r$ be the row reading word of $\SD{w}$. 
We have \[\SD{w}=\young(134,27,56), \quad r=5627134, \quad \text{ and } \quad \P{w}=\young(134,267,5)=\P{r}. \]
\end{example}

In Example~\ref{ex:thm:preserved:452361}, 
 the soliton decomposition coincides with the RS insertion tableau 
of the box-ball system, but 
in Example~\ref{ex:cor:P(w)=P(r):5623714}
 these two tableaux do not coincide. 
In the next section we discuss when~$\SD{w}=\P{w}$.

\section{
When the soliton decomposition and the RS insertion tableau coincide
}
\label{sec:proof of tfae}

In this section, we will prove Theorem~\ref{thm:tfae}. 
One direction of our proof uses 
the following lemma, which was communicated to us by Darij Grinberg. 
\begin{lemma}
\label{lem: Grinberg}
Suppose $S$ is a row-strict tableau, that is, every row is increasing (with no restrictions on the columns). 
Let $r$ be the row reading word of $S$. 
If $\sh S= \sh \P{r}$, 
then $S$ is standard, that is, every column of $S$ is increasing.
\end{lemma}
\begin{proof}
Suppose $S$ is not standard. 
Then $S$ has two adjacent entries in a column which are out of order. 
Indexing our rows from top to bottom and our columns from left to right, this means there is a column (say, column~$c$) for which the entry in some 
row~$k$ is bigger than the entry immediately below it. 
Let $y$ be the entry in the $k$-th row, $c$-th column of $S$, and let $x$ be the entry immediately below it (in the $k+1$-th row, $c$-th column of $S$).

Since $r$ is the row reading word of $S$ and
since each row of $S$ is increasing, 
we can construct a list of $k$ disjoint increasing subsequences of $r$:
The first $k-1$ increasing subsequences of $r$ are the first $k-1$ rows of $S$. 
The $k$-th increasing subsequence starts in row~$k+1$, column~$1$ of~$S$, moving along the same row until we get to column $c$ (with entry $x$), 
then going up to row~$k$ above (which has entry $y$), then continuing to the end of row $k$. 

The length of the $k$-th increasing subsequence is larger (by $1$) than the length of the $k$-th row of $S$. 
So the total number of letters in our list of $k$ disjoint increasing subsequences of $r$ is larger by $1$ than the total length of the first $k$ rows of $S$.
Thus, Greene's theorem (Theorem~\ref{thm:Greene's theorem}) says that the total length of the first $k$ rows of the RS insertion tableau $\P{r}$ of $r$ 
is larger (at least by $1$) than the total length of the first $k$ rows of $S$. 
Therefore, the shape of $S$ is not equal to the shape of $\P{r}$.
\end{proof}

The following theorem 
gives a characterization of permutations whose soliton decompositions are equal to their RS insertion tableaux.

\begin{theorem}\label{thm:tfae} 
Let $w$ be a permutation. Then 
the following are equivalent:
\begin{enumerate}
\item 
\label{tfae:itm:one}
$\SD{w} = \P{w}.$

\item 
\label{tfae:itm:two}
$\SD{w}$ is a standard tableau.

\item 
\label{tfae:itm:three}
The shape of $\SD{w}$ equals the shape of $\P{w}$.
\end{enumerate}
\end{theorem}
\begin{proof}
Certainly \eqref{tfae:itm:one} implies \eqref{tfae:itm:two} and \eqref{tfae:itm:three}. 
We will show that (\ref{tfae:itm:two}) implies (\ref{tfae:itm:one}) and 
(\ref{tfae:itm:three}) implies (\ref{tfae:itm:two}). 

Let $r$ be the row reading word of $\SD{w}$. 
By Corollary~\ref{cor:P(w)=P(r)}, we have 
\begin{equation}
\label{eq:P(w)=P(r)}
\P{w}=\P{r}.  
\end{equation}

First, we show that (\ref{tfae:itm:two}) implies (\ref{tfae:itm:one}). 
Suppose that $\SD{w}$ is a standard tableau $T$.
Since $r$ is the row reading word of $T$, we have $\P{r}=T$ by~\eqref{eq:if r is reading word of T then P(r)=T}.
Combining this equality with~ 
\eqref{eq:P(w)=P(r)}, we get $\P{w}=\P{r} =T=\SD{w}$.

Next, we show that \eqref{tfae:itm:three} implies \eqref{tfae:itm:two}. 
Let $S$ denote $\SD{w}$, and note that $\SD{w}$ is a row-strict tableau by construction. 
Suppose $\sh S =\sh \P{w}$. 
Since $\P{w}=\P{r}$ by~\eqref{eq:P(w)=P(r)}, we have 
$\sh\P{w}=\sh\P{r}$, so $\sh S = \sh\P{w} = \sh \P{r}$.
Because $S$ is a row-strict tableau, the permutation $r$ is the row reading word of $S$, and $\sh S = \sh \P{r}$, 
Lemma~\ref{lem: Grinberg} tells us that $S$ is standard. 
\end{proof}
  
\begin{corollary}
\label{cor:thm:tfae} 
Let $w$ be a permutation. Then 
the following five statements are equivalent:
\begin{enumerate}
\item 
$\SD{w} = \P{w}$.

\item 
$\SD{w}$ is a standard tableau.

\item 
The shape of $\SD{w}$ equals the shape of $\P{w}$.

\item 
For all $k \geq 1$, we have
\[
\localincr_k(w) = \incr_k(w). 
\]

\item 
For all $k \geq 1$, we have
\[
\localdecr_k(w) = \decr_k(w).
\]
The symbols $\localincr_k$ and $\localdecr_k$ are the statistics from localized Greene's theorem (Section~\ref{sec:local Greene's theorem})
and 
$\incr_k$ and $\decr_k$ are the statistics from Greene's theorem (Section~\ref{sec:Greene's theorem}).
\end{enumerate}
\end{corollary}
\begin{proof}
For short, we write $\incr_k \coloneqq \incr_k (w)$, 
$\localincr_k \coloneqq \localincr_k (w)$, 
$\decr_k \coloneqq \decr_k (w)$, and 
$\localdecr_k \coloneqq \localdecr_k (w)$. 
By localized Greene's theorem (Lemma~\ref{lem:local Greene's theorem}), 
\begin{center}
the shape of $\SD{w}$ is 
$(\localincr_1,\localincr_2-\localincr_1, \localincr_3 - \localincr_2, \dots)$ and \\
the shape of the conjugate of 
$\SD{w}$ is 
$(\localdecr_1, \localdecr_2-\localdecr_1, \localdecr_3-\localdecr_2, \dots)$. 
\end{center}
By Greene's theorem (Theorem~\ref{thm:Greene's theorem}), 
\begin{center}
the shape of $\P{w}$ is $(\incr_1, \incr_2-\incr_1, \incr_3-\incr_2, \dots)$ and \\
the shape of the conjugate of 
$\P{w}$ is 
$(\decr_1, \decr_2-\decr_1, \decr_3-\decr_2, \dots)$. 
\end{center}
Combining these facts, 
we conclude that $\sh\SD{w} = \sh \P{w}$ if and only if 
$\localincr_k = \incr_k$ for all $k \geq 1$
if and only if 
$\localdecr_k = \decr_k$ for all $k \geq 1$.
\end{proof}

\begin{example}
Let $w=5623714$. 
From
Examples~\ref{ex:thm:Greene's theorem:5623714} and~\ref{ex:lem:local Greene's theorem:5623714}, we have ${\localincr_2(w) = 5 < 6 = \incr_2(w)}$. 
So all the other items of Corollary~\ref{cor:thm:tfae} 
must also be false.
\end{example}

\section{
Reading words and steady states}
\label{sec:reading words and steady states}

We study the steady-state configurations of a box-ball system. 
The main result of this section 
(Proposition~\ref{prop:t=0 generalization}) 
is a corollary of~\cite[Proof of Lemma 2.1 and 2.3]{LLPS19}.

\subsection{Reading words of standard tableaux}

The permutations which reach their steady state at time $0$ are precisely the row reading words of standard tableaux.

\begin{proposition}
\label{prop:t=0}
A permutation $r$ has steady-state time $0$ if and only if 
$r$ is the row reading word of a standard tableau. 
\end{proposition}

In particular, if $r$ is the row reading word of a standard tableau $T$, then $T = \SD{r}.$ 
In the next section, 
the standard tableau in
Proposition~\ref{prop:t=0} is generalized to 
standard skew tableaux whose rows are weakly decreasing in length.

\subsection{Reading words of standard skew tableaux}

A BBS state can be represented as a \emph{configuration array} containing the integers from 1 to $n$ as follows: scanning the boxes from right to left, each \emph{increasing run} (maximal consecutive increasing string of balls) becomes a row in the array. 
A string of $g$ empty boxes 
 indicates that the next row below should be shifted $g$ spaces to the left.
Note that this array has increasing rows but not necessarily increasing columns; it may be disconnected and it may not have a valid skew shape. 

\begin{proposition}\label{prop:t=0 generalization} 
A BBS configuration 
is in steady state 
if and only if 
its configuration array 
is a standard (possibly disconnected) skew tableau whose 
 rows are weakly decreasing in length.
\end{proposition} 

We will give a proof in Section~\ref{sec:separation condition}. 

\begin{example}
\label{ex:prop:t=0 generalization:5623714}
Let $w= 5623714$, the example we use in Section~\ref{sec:Greene's theorem sec:local Greene's theorem}. 
The following are the box-ball system states from time $t=0$ to $t=4$ and their configuration arrays.
\begin{align*}
t=0 \qquad & 5 \, 6 \, 2 \, 3 \, 7 \, 1 \, 4 \, e \dots 
& {\young(14,237,56)} 
\\[3mm]
t=1 \qquad & e \, e \, 5 \, 6 \, e \, 2 \, 7 \, 1 \, 3 \, 4 \, e \dots 
& 
{\young(:134,:27,56)}
\\[3mm]
t=2 \qquad & e \, e \, e \, e \, 5 \, 6 \, e \, 2 \, 7 \, e \, 1 \, 3 \, 4 \, e \dots
&
{\young(::134,:27,56)}
\\[3mm]
t=3 \qquad & e \, e \, e \, e \, e \, e \, 5 \, 6 \, e \, 2 \, 7 \, e \, e \, 1 \, 3 \, 4 \, e \dots
&
{\young(:::134,:27,56)}
\\[3mm]
t=4 \qquad & e \, e \, e \, e \, e \, e \, e \, e \, 5 \, 6 \, e \, 2 \, 7 \, e \, e \, e \, 1 \, 3 \, 4 \, e \dots
&
{\young(::::134,:27,56)}
\end{align*}
In this box-ball system, all configurations at time $t \geq 1$ are in steady state.
\end{example}

\begin{example}
The following is an example of 
a non-steady-state BBS configuration and 
 its configuration array. 
Note that the configuration array is a standard skew tableau but its rows are not weakly decreasing in length. 
 \begin{align*}
 \qquad & \dots e \, 1 \, 3 \, 7 \, e \, 2 \, 4 \, 6 \, 9 \, e \, e \, 5 \, 8 \, e\dots & {\young(:::58,:2469,137)}
 \end{align*}
\end{example}

\subsection{Separation condition}
\label{sec:separation condition}

A `separation condition' for steady state is given in statement (43) in~\cite{LLPS19}. 
In Lemmas~\ref{lem:separation condition suppose steady state} and~\ref{lem:separation condition suppose two conditions}, 
we reframe this characterization for steady state in terms of our version of the box-ball system. 
Proposition~\ref{prop:t=0 generalization} follows directly from these two lemmas.

\begin{lemma}[Separation condition]
\label{lem:separation condition suppose steady state}
Let a BBS configuration be in steady state. 
Suppose two adjacent solitons $L$ (the left soliton with length $\ell$) and $R$ (the right soliton) are separated 
by $g$ empty boxes, where $g<\ell$.
Then, 
for $i=1,2,\dots,\ell-g$, 
\begin{center}
the $i$-th smallest ball of the right soliton $R$ is smaller than \\
the $(i+g)$-th smallest ball of the left soliton $L$.
\end{center}
\end{lemma}
\begin{proof}
We apply one BBS move to the configuration via the carrier algorithm.
Suppose \linebreak $L=L_1 L_2 \dots L_\ell$ and $R=R_1 R_2 \dots R_r$ are the two leftmost solitons. 

Our initial setup with $n$ copies of $e$ in the carrier is 
\[
\underbracket{\, ee \cdots e \,}_{\text{carrier}} 
{\, L_1 \, \dots \, L_\ell \,}
\overbrace{e \dots e}^{\mathclap{\text{$g$ copies}}}
{\, R_1 \, \dots \, R_r \,}
  \dots 
\]
First, 
we simply insert 
 $L_1,\dots,L_\ell$ into the carrier. 
Since $L$ is increasing, 
 each time we insert a ball of $L$, we 
 eject a copy of $e$. 
We get 
\begin{equation}
\label{eq:separation_condition1}
\overbrace{e \dots e}^{\mathclap{\text{$\ell$ copies}}}
\underbracket{\, L_1 \, \cdots \, L_\ell ~ ee \cdots e \,}
~
\overbrace{e \dots e}^{\mathclap{\text{$g$ copies}}}
{\, R_1 \, \dots \, R_r \,} 
\dots 
\end{equation}
Next, we insert the $g$ copies of $e$ into the carrier and eject $L_1,\dots,L_g$:
\[
e \dots e
\, 
\underbrace{L_1 \, \dots \, L_g}_{\text{first $g$ balls}}
\,
\underbracket{
\,
\overbrace{L_{g+1} \, \cdots L_\ell}^{\mathclap{\text{$\ell-g$ balls}}} ~ ee \cdots e \,}
{\, R_1 \, \dots \, R_r \,}
\dots 
\]
Since we started with a steady-state configuration, 
the left soliton $L$ must stay intact at the end of the carrier algorithm. 
So, for each $i=1,\dots,\ell-g$,
as we insert
$R_{i}$, 
we must eject 
$L_{g+i}$,
and get
\[
e \dots e \, 
L_1 \, \dots \, L_g
 \, 
 \underbrace{L_{g+1} \, \dots \, L_\ell}_{\text{$\ell-g$ balls}} 
\underbracket{R_{1} \, \cdots \, R_{\ell-g} ~ ee \cdots e \, }
{R_{\ell-g+1}\dots R_r}
 \dots 
\]
So we must have 
$R_i<L_{g+i}$ for $i=1,2,\dots,\ell-g$, 
as needed. 

After we insert the rest of the elements of $R$ into the carrier, 
we have 
\[
e \dots e \, 
{L_1 \, \dots \, L_\ell}
~
\overbrace{e e \dots e}^{\substack{r-\ell+g \\ \text{copies}}} 
~
\underbracket{\, R_{1} \, \cdots \, R_r ~ ee \cdots e \,} \dots
\]
If we have a third soliton located to the right of $R$, we would be in the same situation as \eqref{eq:separation_condition1}. 
We then repeat the same process for the rest of the solitons and arrive at the same conclusion.
\end{proof}

\begin{lemma}[Sufficient condition for steady state]
\label{lem:separation condition suppose two conditions}
Suppose a BBS configuration $w$ satisfies the following.
\begin{enumerate}
\item \label{lem:separation condition suppose two conditions:1}
The configuration array of $w$ has rows of weakly decreasing length.
\item \label{lem:separation condition suppose two conditions:2}
The configuration array of $w$ is standard; 
that is, 
if two adjacent maximal consecutive increasing blocks $L$ (the left block with length $\ell$) and $R$ (the right block) of $w$ 
are separated by $g$ empty boxes such that $g<\ell$, 
then, 
for $i=1,2,\dots,\ell-g$, 
\begin{center}
the $i$-th ball of the right block $R$ is smaller than
\\
the $(i+g)$-th ball of the left block $L$.
\end{center}
\end{enumerate}
Then $w$ is in steady state.
\end{lemma}
\begin{proof}
Suppose $w$ is the configuration at time $t$. 
We apply the carrier algorithm to get the configuration at time $t+1$. 
Suppose $L=L_1 L_2 \dots L_\ell$ and $R=R_1 R_2 \dots R_r$ are the two leftmost increasing runs (maximal consecutive increasing blocks of balls).

Prior to applying the carrier algorithm, we have 
\[
  \underbracket{\, ee \cdots e \,}_{\text{carrier}} 
  \underbrace{L_1\dots L_\ell}_{\text{first run}}
  \overbrace{e \dots e}^{\mathclap{\text{$g$ copies}}}
  \underbrace{R_1\dots R_r}_{\text{second run}} \dots 
\]
First, we insert each of $L_1,\dots,L_\ell$ into the carrier and eject an $e$ each time. 
We get 
\begin{equation}
\label{eq:separation_condition2}
\overbrace{e \dots e}^{\mathclap{\text{$\ell$ copies}}}
\underbracket{\, L_1\cdots L_\ell ~ ee \cdots e \,}
~
\overbrace{e \dots e}^{\mathclap{\text{$g$ copies}}}
{R_1\dots R_r}
\end{equation}
Next, we insert the $g$ copies of $e$ into the carrier and eject $L_1,\dots,L_g$. 
There are two cases: either (a) $g \geq \ell$ or (b) $g < \ell$. 

\begin{enumerate}[(a)]
\item First, suppose that $g \geq \ell$. 
Then all of $L_1,\dots,L_\ell$ are ejected and the carrier is now empty:
\[
e \dots e \, 
\underbrace{L_1\dots L_\ell}_{\text{first run}}
\overbrace{e\dots e}^{g-\ell}
\underbracket{\, ee \cdots e \,}
\underbrace{R_1\dots R_r}_{\text{second run}} \dots 
\]
We proceed by inserting $R_1,\dots,R_r$ into the carrier.
Since $R$ is increasing, we eject $r$ copies of $e$'s:
\[
e \dots e \, 
{L_1\dots L_\ell}
\overbrace{e\dots e}^{g-\ell}
\overbrace{e\dots e}^{r}
\underbracket{\, R_1\cdots R_r ~ee \cdots e \,} \dots 
\]

\item Second, suppose $g < \ell$. After $L_1,\dots,L_g$ are ejected, we have 
\[
e \dots e
\underbrace{L_1\dots L_g}_{\text{first $g$ balls}}
\underbracket{\, 
\overbrace{L_{g+1} \cdots L_\ell}^{\mathclap{\ell-g \text{ balls}}} ee \cdots e \, }
\underbrace{R_1\dots R_r}_{\text{second run}} \dots 
\]
We proceed by inserting $R_1,\cdots,R_r$ into the carrier. 
We have
$\ell \leq r$ 
by assumption part~\eqref{lem:separation condition suppose two conditions:1}
and 
$R_i<L_{g+i}$ for $i=1,2,\dots,l-g$
by assumption part~\eqref{lem:separation condition suppose two conditions:2}. 
Therefore, as we insert~$R_1,\dots,R_{\ell-g}$,
we must eject $L_{g+1},\dots, L_\ell$, and we get 
\[
e \dots e \, 
{L_1\dots L_g}
 \underbrace{L_{g+1}\dots L_\ell}_{\ell-g \text{ balls}} 
\underbracket{\, R_{1} \cdots R_{\ell-g} ~ ee \cdots e \,}
{R_{\ell-g+1}\dots R_r}
 \dots 
\]
After we insert the rest of the elements of $R$ into the carrier, 
we have 
\[
e \dots e \,
{L_1\dots L_\ell}
~
\overbrace{e e \dots e}^{\mathclap{r-\ell+g}}
~
\underbracket{ \, R_{1} \cdots R_r ~ ee \cdots e \, } \dots
\]
\end{enumerate}

In both cases, at time $t+1$ there are at least $r-\ell + g$ empty boxes to the right of $L$. 
Since~$\ell \leq r$, we have $g \leq r - \ell +g$, 
so there are at least as many empty boxes to the right of $L$ as at time $t$.
Furthermore, the increasing run $L$ stays together.

If we have a third increasing run $S=S_1\dots S_s$ to the right of $R$ (with a gap of $g'$ empty boxes), we would be in the same situation as \eqref{eq:separation_condition2}. 
After inserting the elements of $S$ into the carrier, we would have 
\[
e \dots e \,
{L_1\dots L_\ell}
~
\overbrace{e e \dots e}^{\mathclap{r-\ell+g}}
~
{R_1\dots R_r}
~
\overbrace{e e \dots e}^{\mathclap{s-r+g'}}
~
\underbracket{S_{1} \cdots S_s ~ ee \cdots e} \dots
\]
Again, there are at least as many empty boxes to the right of $R$ at time $t+1$
than at time $t$, and~$R$ stays together.

At the end of the carrier algorithm, the increasing runs stay together, their order stays the same, and the gap of empty boxes between each pair of adjacent sequences is at least as large as at time $t$. 
The new configuration satisfies both part~\eqref{lem:separation condition suppose two conditions:1} and \eqref{lem:separation condition suppose two conditions:2} of the assumption. 
By induction, subsequent carrier algorithm applications leave the order of the increasing runs unchanged, so these increasing runs are in fact solitons.
\end{proof}

By the two lemmas above, we have 
Proposition~\ref{prop:t=0 generalization}: 
a box-ball configuration is in steady state if and only if
(1) its configuration array has rows of weakly decreasing length and 
(2) each column of the configuration array is increasing.

\section{A recording tableau giving 
\textit{n--}3
steady-state time} 
\label{sec5}

In this section, 
we prove 
Theorem~\ref{thm:n minus 3}, which states 
that all permutations in $S_n$ with a certain recording tableau have box-ball steady-state time $n-3$. 
We conjecture that all other permutations in $S_n$ have steady-state time smaller than $n-3$ (Conjecture~\ref{conj: n minus 3}). 

Theorem~\ref{thm:n minus 3} turns out to be a special case of a general phenomenon, which is proven in a sequel to this paper~\cite{sumry21}: 
if two permutations have the same recording tableau, then they have the same BBS steady-state time (Conjecture~\ref{conj: same Q implies same t}).

\subsection{A recording tableau giving \textit{n--}3 steady-state time}

\begin{definition}\label{def: Q_0}
If $n \geq 5$, let $\Qmaxtime$ denote the tableau $$\begin{ytableau}
 1 & 2 &\none[\hdots]&\scalebox{.6}{$n-2$}&\scalebox{.6}{$n-1$}\\
 3 & 4\\
 n
 \end{ytableau}.$$
Let $\SnQmaxtime$ be the set of permutations $w \in S_n$ such that its recording tableau $\Q{w}$ is equal to $\Qmaxtime$.
\end{definition}

\pagebreak 
\begin{example}
\label{example:tableau n-3}
For $n=5$, the five permutations of $\SnQmaxtime$ are the following.
\begin{multicols}{5}
 45132
 
 25143
 
 35142
 
 45231
 
 35241
\end{multicols}

For $n=6$, the sixteen permutations of $\SnQmaxtime$ are as follows.

\begin{center}
\begin{tabularx}{\textwidth}{  
  >{\centering\arraybackslash}X 
  >{\centering\arraybackslash}X 
  >{\centering\arraybackslash}X
  >{\centering\arraybackslash}X 
  >{\centering\arraybackslash}X 
  >{\centering\arraybackslash}X 
  >{\centering\arraybackslash}X 
  >{\centering\arraybackslash}X}
451362&
251463&
351462&
452361&
352461&
561243&
261354&
361254\\
461253&
561342&
261453&
361452&
461352&
562341&
362451&
462351
\end{tabularx}
\end{center}
Note that one of our running examples, $452361$, 
is in $S_6 (\widehat{Q})$.
As illustrated in Figure~\ref{fig:intro:452361:t=0 to t=4}, its 
steady-state time is $3=6-3$. 
\end{example}

 \begin{remark}
It follows from Definition~\ref{def: Q_0} that the RS algorithm induces a bijection from $\SnQmaxtime$ 
onto the set of standard tableaux of shape $(n-3,2,1)$, so $\SnQmaxtime$ is counted by the sequence~\cite[A077415]{OEIS}.
\end{remark}

The rest of this section is devoted to proving Theorem~\ref{thm:n minus 3}, which states that every permutation in $\SnQmaxtime$ has steady-state time $n-3$.

\subsection{Lemmas for Theorem~\ref{thm:n minus 3}}

\begin{lemma}
\label{lem:w is not the union of increasing subsequences}
Let $n \geq 5$, and 
suppose $w \in \SnQmaxtime$. 
Then $w$ is \emph{not} the union of two increasing subsequences.
\end{lemma}
\begin{proof}
The recording tableau of $w$ is equal to $\Qmaxtime$, which has height $3$. 
Therefore, the RS partition of $w$ has three parts.
By Greene's theorem 
(Theorem~\ref{thm:Greene's theorem}),
 $w$ is not the union of two increasing subsequences.
\end{proof}

\begin{lemma}
\label{lem:inverse RS}
Let $n \geq 5$, and 
suppose $w=w_1 w_2 \dots w_n \in \SnQmaxtime$. Then $w$ satisfies the following.

\begin{enumerate}
\item
\label{lem:inverse RS:increasing subsequence}
$w_3 < w_4 < \dots < w_{n-1}$

\item 
\label{lem:inverse RS:wn<w2}
$w_n <w_2$

\item 
\label{lem:inverse RS:w1<w2}
$w_1<w_2$

\item
\label{lem:inverse RS:w3<w1}
$w_3<w_1$ 

\item 
\label{Q-0.2}
\label{lem:inverse RS:w3<w2}
$w_3 < w_2$

\item \label{Q-0.3}
\label{lem:inverse RS:w4<w2}
$w_4 < w_2$

\end{enumerate}
\end{lemma}

\begin{proof}
\par Since $w\in \SnQmaxtime$, 
the recording tableau of $w$ is equal to $\Qmaxtime.$ 
We will use the inverse RS algorithm\footnote{For definition of the inverse RS algorithm, see, for example, the textbook~\cite[Section 3.1]{Sag01}.} to construct $w.$ 
Let $P = \P{w}$ and $Q = \Q{w}$. 
Denote the entries in the top row of~$P$ by $a_1,a_2,\dots, a_q$ (where $q=n-3$), the second row of $P$ by $b_1$ and $b_2$, and the entry in the third row of $P$ by $c_1$.
Hence, the starting pair $P$ and $Q$ is
$$
P = \begin{ytableau}
 a_1 & a_2 & a_3 & a_4 &\none[\hdots]&a_{\scalebox{.44}{${q-1}$}}&a_q\\
 b_1 & b_2\\
 c_1\end{ytableau} \hspace{5em} Q = \begin{ytableau}
 1 & 2 & 5 & 6 &\none[\hdots]&\scalebox{.55}{$n-2$}&\scalebox{.55}{$n-1$}\\
 3 & 4\\
 n\end{ytableau}
$$
Since $P$ is standard, we know that $b_1<c_1$. 
The other entry $b_2$ in the second row is larger than~$b_1$. 
If
$b_2<c_1$, 
let $b_y$ equal $b_2$. Otherwise, let $b_y$ be $b_1$. 
In other words, we let $b_y$ denote the largest element in the second row which is smaller than $c_1$. 
Similarly, let $a_x$ denote the largest element in the first row which is smaller than $b_y$.
The first step of the inverse RS algorithm tells us that~$w_n=a_x$.

After the first step in the inverse RS algorithm, we get the pair of tableaux
\[
P_{n-1} = \begin{ytableau}
 \alpha_1 & \alpha_2 & \alpha_3 & \alpha_4 &\none[\hdots]&\alpha_{\scalebox{.44}{${q-1}$}}&\alpha_q\\
 \beta_1 & \beta_2\end{ytableau} \hspace{5em} Q_{n-1} = \begin{ytableau}
 1 & 2 & 5 & 6 &\none[\hdots]&\scalebox{.55}{$n-2$}&\scalebox{.55}{$n-1$}\\
 3 & 4\end{ytableau}.
\]

We now pause to observe two facts that will be referenced at the end of this proof. 
First, note that $P_{n-1}$ is standard 
by definition of the inverse RS algorithm.  
Thus,
\begin{equation}
\label{eq:alpha1 alpha2 dots alphaq is increasing}
\alpha_1, \alpha_2, \dots, \alpha_q \text{ is increasing}.
\end{equation}

Second, we note that 
\begin{equation}
\label{eq:ax<beta2}
a_x < \beta_2,
\end{equation} as we now explain. 
Recall that $w_n=a_x$, so, using the original RS algorithm, we insert $a_x$ into~$P_{n-1}$ to get $P$. 
Since row $1$ of $P_{n-1}$ and row $1$ of $P$ have the same size, 
we know that $a_x$ bumps a number in row $1$ of $P_{n-1}$ to row $2$. 
Let 
\begin{center}
$a_i$ denote the smallest entry in row $1$ of $P_{n-1}$ which is greater than $a_x$. 
\end{center}
The RS algorithm replaces $a_i$ with $a_x$ and bumps $a_i$ to row $2$.  
Since row $2$ of $P_{n-1}$ and row $2$ of~$P$ have the same size, 
we know that $a_i$ bumps a number in row $2$ of $P_{n-1}$. 
So $a_i$ must be smaller than $\beta_2$. Since $a_x<a_i$, we have $a_x < \beta_2$. This concludes our explanation for~\eqref{eq:ax<beta2}.

We also note that
\begin{align}
\label{eq:beta1<beta2}\beta_1 &<\beta_2,\\
\label{eq:alpha1<beta1}\alpha_1 &<\beta_1, ~\text{ and } \\
\label{eq:alpha2<beta2}\alpha_2 &<\beta_2, 
\end{align}
since $P_{n-1}$ is standard. We will reference these inequalities at the end of this proof.

If $n > 5$, the numbers $n-1, n-2, \dots, 6,5$ 
are in the first row of $Q$, so the next steps in the inverse RS algorithm are 
to remove elements $\alpha_q,\alpha_{q-1},\dots,\alpha_4,\alpha_3$ from $P_{n-1}$, in that order. 
Hence, the last $n-4$ letters of $w$ are $\alpha_3,\alpha_4,\dots,\alpha_{q-1},\alpha_{q},a_x$.

The new pair of tableaux is 
\[
P_4 = \begin{ytableau}
 \alpha_1 & \alpha_2\\
 \beta_1 & \beta_2\end{ytableau} \hspace{5em} Q_4 = \begin{ytableau}
 1 & 2 \\
 3 & 4\end{ytableau}.
\]
Note that $4$ is the bottom right corner of $Q_4$. 
Since $\alpha_2 < \beta_2$ by~\eqref{eq:alpha2<beta2}, we know that $\alpha_2$ is the largest element in row~1 of $P$ which is smaller than $\beta_2$. 
So $w_4=\alpha_2$, 
and  
the last $n-3$ letters of~$w$ are $\alpha_2,\alpha_3,\alpha_4,\dots,\alpha_q,\alpha_{q-1},a_x$.

The new pair of tableaux is 
\[
P_3 = \begin{ytableau}
 \alpha_1 & \beta_2\\
 \beta_1 \end{ytableau} \hspace{5em} Q_3 = \begin{ytableau}
 1 & 2 \\
 3 \end{ytableau}.
\]
Note that $3$ is in the second row of $Q_3$. 
We know from~\eqref{eq:beta1<beta2} that $\beta_2$ is larger than $\beta_1$, so $\alpha_1$ is the largest element in row~1 smaller than $\beta_1$. Thus, $w_3=\alpha_1$. 
So the last $n-2$ letters of $w$ are~$\alpha_1,\alpha_2,\alpha_3,\alpha_4,\dots,\alpha_q,\alpha_{q-1},a_x$. 
The new pair of tableaux is 
\[
P_2 = \begin{ytableau}
 \beta_1 & \beta_2\end{ytableau} \hspace{5em} Q_2 = \begin{ytableau}
 1 & 2 \end{ytableau}.
\]
We then remove $\beta_2$ and $\beta_1$ from $P_2$, in that order. 

Therefore, 
\[w =
\underbrace{\beta_1\beta_2}_{\text{increasing}}\underbrace{\alpha_1 \alpha_2 \alpha_3 \alpha_4 \dots \alpha_{q-1}\alpha_q}_{\text{increasing}} a_x. 
\] 
We now have all the necessary information to prove all parts of the lemma.
\begin{enumerate}
\item
The subsequence 
$w_3$, $w_4$, $\dots$, $w_{n-1}$ is 
increasing
because it is equal to the sequence 
$\alpha_1$, 
$\alpha_2$, 
$\dots$, 
$\alpha_q$, 
which is increasing due to~
\eqref{eq:alpha1 alpha2 dots alphaq is increasing}. This proves part~\eqref{lem:inverse RS:increasing subsequence}.

\item 
We have $w_n < w_2$ from~\eqref{eq:ax<beta2}, since $w_n=a_x$ and $w_2=\beta_2$. 
This proves part~\eqref{lem:inverse RS:wn<w2}.

\item 
We have $w_1<w_2$ from~\eqref{eq:beta1<beta2}, since $w_1 = \beta_1$ and $w_2 = \beta_2$. This proves part~\eqref{lem:inverse RS:w1<w2}.

\item 
We have $w_3<w_1$ 
from~\eqref{eq:alpha1<beta1}, since $w_3=\alpha_1$ and $w_1=\beta_1$. 
This proves part~\eqref{lem:inverse RS:w3<w1}. 

\item 
We have $w_3<w_2$ since $w_1<w_2$ and $w_3<w_1$. 
This proves part~\eqref{lem:inverse RS:w3<w2}. 

\item 
We have $w_4<w_2$ 
from~\eqref{eq:alpha2<beta2}, since $w_4=\alpha_2$ and $w_2=\beta_2$.
This proves part~\eqref{lem:inverse RS:w4<w2}. 
\qedhere
\end{enumerate}
\end{proof}

\begin{lemma}
\label{lem: these five are the only cases}
Suppose $w=w_1\dots w_n \in S_n(\Qmaxtime)$. 
\begin{enumerate}
\item 
Either $w_n=1$ or $w_3=1$.
 
\item 
If $w_3=1$, then $w_1=2$, $w_4=2$, or $w_n=2$.

\item 
If $w_3=1$ and $w_1=2$, then $w_4=3$ or $w_n=3$.
\end{enumerate}
\end{lemma}
\begin{proof}
\begin{enumerate}
\item 
Suppose $w_n \neq 1$. Since both $w_1,w_2$ and 
$w_3,\dots,w_{n-1}$ are increasing subsequences by Lemma~\ref{lem:inverse RS}\eqref{lem:inverse RS:w1<w2},\eqref{lem:inverse RS:increasing subsequence}, 
either $w_1=1$ or $w_3=1$. 
Since $w_3 < w_1$ by
Lemma~\ref{lem:inverse RS}\eqref{lem:inverse RS:w3<w1}, 
we must have $w_3=1$.

\item 
Assume $w_3 = 1$. 
We will show that $w_2\neq 2$ and that none of $w_5,\dots,w_{n-1}$ is equal to $2$ (hence $w_1=2$, $w_4=2$, or $w_n=2$). 
Since 
$w_3<w_2$ and $w_4<w_2$ 
by Lemma~\ref{lem:inverse RS}\eqref{lem:inverse RS:w3<w2},\eqref{lem:inverse RS:w4<w2} 
and since $w$ is a permutation, 
we must have $2 < w_2$. 
Similarly, since $w_3<w_4<w_5<\dots<w_{n-1}$ 
by Lemma~\ref{lem:inverse RS}\eqref{lem:inverse RS:increasing subsequence}
and since $w$ is a permutation, each of $w_5,\dots,w_{n-1}$ must be larger than $2$. 

\item 
Suppose $w_3=1$ and $w_1=2$. 
We will prove that $w_2\neq 3$ and none of $w_5,\dots,w_{n-1}$ 
is equal to $3$ (hence $w_4=3$ or $w_n=3$). 
Since $w_n \notin \{ 1,2\}$, we have $2< w_n$. 
By Lemma~\ref{lem:inverse RS}\eqref{lem:inverse RS:wn<w2}, we have $w_n < w_2$. 
So $2< w_n < w_2$, which implies that $w_2$ is larger than~$3$ (since $w$ is a permutation). 
Similarly, 
since $w_3 < w_4$ 
by Lemma~\ref{lem:inverse RS}\eqref{lem:inverse RS:increasing subsequence} 
and $w_1=2$, we must have~$2<w_4<w_5<\dots<w_{n-1}$ 
by Lemma~\ref{lem:inverse RS}\eqref{lem:inverse RS:increasing subsequence}. 
So each of $w_5,\dots,w_{n-1}$ is larger than $3$ (since $w$ is a permutation).
\qedhere
\end{enumerate}
\end{proof}

\subsection{Proof of Theorem~\ref{thm:n minus 3}}

\begin{theorem}
\label{thm:n minus 3}
If $n \geq 5$, 
 every permutation in $\SnQmaxtime$ has steady-state time $n-3.$
\end{theorem}
\begin{proof}
Suppose $w=w_1\dots w_n \in S_n(\Qmaxtime)$ is the box-ball configuration at time $0$. 
We will show that $w$ first reaches steady state at time $t=n-3$. 

Let $j$ be the smallest number in $\{3,4,\dots,n-1 \}$ such that $w_n < w_j$. 
We claim that the box-ball configuration at time $t=1$ is 
\begin{equation}
\label{thm:n minus 3 copy:configuration at t=1}
e \, e \, \underbrace{w_1 \, w_2}_{\substack{\text{increasing}\\ \text{block}}} \, 
\overbrace{e \, e \, e \, \dots \, e}^{\substack{n-5\\ \text{ copies}}} 
\, x \, \underbrace{1 \, y_1 \, y_2 \, \dots \, y_{n-4}}_{\text{increasing block}},
\end{equation}
where $x=w_j$, there are $(n-5)$ copies of $e$ between $w_2$ and $x$, 
and $y_1 < y_2 < \dots < y_{n-4}$.

To prove this claim, 
 consider the following cases.  
Due to 
 Lemma~\ref{lem: these five are the only cases}, these five cases cover all possibilities. 
\begin{enumerate}
\item   
$w_n=1$
\item $w_3=1$ and $w_n=2$
\item 
$w_3=1$, $w_1=2$, and $w_n = 3$ 

\item 
$w_3=1$, $w_1 =2$, and $w_4=3$

\item 
$w_3=1$ and $w_4 =2$
\end{enumerate}

First, suppose $w_n=1$. 
Lemma~\ref{lem:inverse RS} tells us that 
 $w_3$ is smaller than each $w_i$ except for $w_n=1$, so we must have $w_3=2$ and $j=3$: 
\[
w_1 \, w_2 
\underbrace{w_3}_{2} w_4 \, w_5 \, \dots \, w_{n-1} \underbrace{\, w_n \,}_{1}. \qquad \qquad\qquad 
\] 
Since $w_1<w_2$ and $w_4 < w_5 < \dots < w_{n-1}$ and since $w_4 < w_2$, 
applying one box-ball move to~$w$ results in the configuration 
\[
\qquad \qquad e \, e \, w_1 \, w_2 \, \overbrace{e \, e \, e \, \dots \, e}^{n-5} \underbrace{w_3}_{x} \, 1 \, w_4 \, w_5 \, \dots \, w_{n-1}
\]
where 
there are $(n-5)$ copies of $e$ between $w_2$ and $x=w_3=2$. 


Second, suppose $w_3=1$ and $w_n=2$: 
\[
w_1 \, w_2 
\underbrace{w_3}_{1} w_4 \, w_5 \, \dots \, w_{n-1} \underbrace{ w_n}_{2}. \qquad \qquad \quad
\] 
Since $w_1<w_2$ and $w_4 < w_5 < \dots < w_{n-1}$ and since $w_4 < w_2$, 
applying one box-ball move to~$w$ results in the configuration 
\[
\qquad \qquad \qquad
e \, e \, w_1 \, w_2 \, \overbrace{e \, e \, e \, \dots \, e}^{n-5\text{ copies}} \underbrace{\, w_4 \,}_{x} 1 \, 2 \, w_5 \, w_6 \, \dots \, w_{n-1}
\]
where there are $(n-5)$ copies of $e$ between $w_2$ and $x=w_4$. 
In this case, $w_3=1$ is not bigger than $w_n=2$, 
but $w_4$ must be bigger than $w_n=2$ since $w_4 \notin \{ 1,2\}$, 
so $j=4$.


Third, suppose $w_3=1$ and $w_1=2$ and $w_n=3$. 
Lemma~\ref{lem:inverse RS} tells us that 
 $w_4$ is smaller than each of the $w_i$ (except for $w_3=1$, $w_1=2$, and $w_n=3$), so $w_4$ must be $4$:
\[
\underbrace{w_1}_{2} \, w_2 
\underbrace{w_3}_{1} \underbrace{w_4}_{4} \, w_5 \, \dots \, w_{n-1} \underbrace{ w_n}_{3}.
\qquad \qquad \qquad
\] 
Using the same reasoning as in the previous two cases, 
applying one box-ball move to $w$ results in the configuration 
\[
\qquad \qquad \qquad
e \, e \, \underbrace{w_1}_{2} \, w_2 \, \overbrace{e \, e \, e \, \dots \, e}^{n-5 \text{ copies}} \underbrace{\, w_4 \,}_{x} 1 \, w_n \, w_5 \, w_6 \, \dots \, w_{n-1}
\]
where there are $(n-5)$ copies of $e$ between $w_2$ and $x=w_4$. 
In this case, $j=4$ since $w_3=1$ is not larger than $w_n=3$ but $w_4=4$ is.


Finally, 
suppose
we have one of the last two cases, so $w_3=1$ 
and $w_4 < w_n$: 
\[
w_1 \, w_2 
\underbrace{w_3}_{1} w_4 \, w_5 \, \dots \, w_{n-1} \underbrace{\, w_n \,}_{\substack{\text{larger}\\ 
\text{than } w_4} } 
\qquad \qquad \qquad \qquad \qquad
\] 
Since $w_1<w_2$ and $w_4 < w_5 < \dots <
w_{j-1} < w_n < w_j < \dots
< w_{n-1}$ and since $w_4 < w_2$, 
applying one box-ball move to $w$ results in the configuration 
\[
\qquad \qquad \qquad \qquad \quad
e \, e \, w_1 \, w_2 \, \overbrace{e \, e \, e \, \dots \, e}^{n-5 \text{ copies}} \underbrace{\, w_j \,}_{x} 1 \, w_4 \, w_5 \, \dots \, w_{j-1} \, w_n \, w_{j+1} \, \dots \, w_{n-1}
\]
where there are $(n-5)$ $e$'s between $w_2$ and $x=w_j$. 
In this case, $j \geq 5$ since $w_4$ is smaller than $w_n$. 
This concludes the proof of our claim that the box-ball configuration at time $t=1$ is as given in~\eqref{thm:n minus 3 copy:configuration at t=1}. 


Now we perform another box-ball move to reach the configuration at $t=2$. 
If $n>5$, in the configuration at $t=2$, there are $(n-6)$ $e$'s between $w_2$ and $x$: 
\[
e \, e \, e \, e \, w_1 \, w_2 \,
\overbrace{e \, e \, \dots \, e}^{\substack{n-6\\\text{copies}}} 
\,\, 
x
\,\,
\overbrace{
e \, e \, \dots e}^{\substack{n-4\\\text{copies}}} \, 
\underbrace{ 1 \, y_1 \, y_2 \, \dots \, y_{n-4}}_{\text{increasing block}}.
\] 
In fact, at each BBS move, the increasing sequence $w_1,w_2$ moves together two spaces to the right, the singleton $x$ moves one space to the right, and the increasing sequence $1, y_1 , y_2 \dots , y_{n-4}$ moves $n-3$ spaces to the right.
So the number of $e$'s between $w_2$ and $x$ decreases by $1$ after each BBS move. 
The configuration at $t=n-4$ is 
\[
\dots \, e \, e \, e \, w_1 \, w_2 \, x \, e \, \, e \, e \, \dots \, e \, e \, e \, \underbrace{1 \, y_1 \, y_2 \, \dots \, y_{n-4}}_{\text{increasing block}}.
\]

We claim that \[x<w_2 ,\]
which we now prove. 
Recall that $x=w_j$, where $j$ is the smallest number in $\{3,4,\dots,n-1 \}$ such that $w_n<w_j$. 
If $w_2 < w_j$, then $w_1< w_2 < w_j < w_{j+1} < \dots < w_{n-1}$ 
and the remaining $w_i$'s form two increasing subsequences of $w$ whose union is $w$. 
This contradicts Lemma~\ref{lem:w is not the union of increasing subsequences}, so indeed $x < w_2$. 

Since $x<w_2$, we have either $x<w_1<w_2$ or $w_1 < x < w_2$. 
If $x<w_1<w_2$, then the configuration at $t=n-3$ is 
\[
\dots \, e \, e \, w_1 \, \underbrace{x \, w_2}_{\substack{\text{increasing}\\\text{block}}} 
\, e \, e \, \dots \, e \, e \, \, \underbrace{1 \, y_1 \, y_2 \, \dots \, y_{n-4}}_{\text{increasing block}}.
\]

If $w_1 < x < w_2$, 
then the configuration at $t=n-3$ is 
\[
\dots \, e \, e \, w_2 \, \underbrace{w_1\, x}_{\substack{\text{increasing}\\\text{block}}} \, e \, e \, \dots \, e \, e \, \, \underbrace{1 \, y_1 \, y_2 \, \dots \, y_{n-4}}_{\text{increasing block}}. 
\]

Either way, 
the configuration array at $t=n-3$ is a standard skew tableau whose rows have length $n-3$, $2$, and $1$. 
By Proposition~\ref{prop:t=0 generalization}, 
the configuration at $t=n-3$ is in steady state. 

The configuration at $t=n-4$ is not yet in steady-state, as the relative positions of $w_1, w_2$, and~$x$ in the configuration at $t=n-4$ differ from the configuration at $t=n-3$. 
Therefore,~${t=n-3}$ is the minimum steady-state time of $w$.
\end{proof}

\section{Knuth moves}
\label{sec:knuth}

We study how types of Knuth moves (Definition~\ref{def: Knuth moves}) play a role in a box-ball system. 
In Section~\ref{sec: knuth paths}, 
 we prove that
 a non-$K_B$ Knuth move 
preserves the shape of a soliton decomposition 
and that 
 a $K_B$ move 
 changes it 
 (Theorem~\ref{thm:knuth paths}). 
In Section~\ref{sec: t=1}, we prove that every permutation which is one non-$K_B$ Knuth move from a row reading word has steady-state time~$1$~(Theorem~\ref{thm:t=1}).

\subsection{Soliton decompositions are preserved by certain
Knuth moves}
\label{sec: knuth paths}

Using the localized version of Greene's Theorem 
given in Section~\ref{sec:local Greene's theorem}, 
we prove a partial characterization of 
 the shape of 
 $\SDself$ 
in terms of types of Knuth moves.

\begin{theorem}
\label{thm:knuth paths}
Suppose $\pi$ and $w$ are two permutations in the same Knuth equivalence class. 
\begin{enumerate}
\item 
\label{thm:knuth paths:KB}
If $\pi$ and $w$ are related by a sequence of Knuth moves containing an odd number of $K_B$ moves, then $ \SD{\pi}
\neq
\SD{w}$. 

\item \label{thm:knuth paths:non-KB}
If $\pi$ and $w$ are related by a sequence of non-$K_B$ Knuth moves, then $\sh\SD{\pi}=
\sh\SD{w}$. 
\end{enumerate}
\end{theorem} 
\begin{proof} 
To prove part~\eqref{thm:knuth paths:KB}, 
we observe that a $K_B^+$ move 
decreases the number of descents by $1$,
and a $K_B^-$ move 
 increases the number of descents by $1$. 
Since 
the height the partition $\sh \SD{w}$
 is equal to 
 \begin{equation*}
 \text{
 $\localdecr_1(w) =  1 + |\{$descents of $w\}|$ 
 }
 \end{equation*}
 by Lemma~\ref{lem:local Greene's theorem},
it follows that applying an odd number of $K_B$ moves to $w$ changes $\sh \SD{w}.$

To prove part~\eqref{thm:knuth paths:non-KB}, 
suppose $x,y\in S_n$ are related by a proper $K_1$ or proper $K_2$ move.
Due to Lemma~\ref{lem:local Greene's theorem}, 
it suffices to prove that $\localdecr_k(x) = \localdecr_k(y)$ for all $k$.
This breaks down into two main cases: 
case~\eqref{thm:knuth paths:non-KB:K1}, where $y= K_1^+(x)$,  
and case~\eqref{thm:knuth paths:non-KB:K2}, where $y=K_2^+(x)$. 
These further divide into the following subcases, where $a<b<c$ in all cases: 
 \begin{enumerate}[i.]
\item
\label{thm:knuth paths:non-KB:K1}
\begin{enumerate}
\item 
\label{thm:knuth paths:non-KB:K1:bacd}
$y=\cdots bca$ \qquad or \qquad $y=\cdots bcad\cdots$ with $c<d$ \\
$x=\cdots bac$ \qquad
or \qquad
$x=\cdots bacd\cdots$

\item 
$y=\cdots bca$
\qquad 
or
\qquad 
$y=\cdots bcaa' $ with $a' <a$\\
$x= \cdots bac$
\qquad 
or
\qquad 
$x= \cdots baca'$
\end{enumerate}

\item
\label{thm:knuth paths:non-KB:K2}
\begin{enumerate}
\item 
$y=cab\cdots$
\qquad 
or
\qquad 
$y=\cdots dcab\cdots$ with $c<d$
\\
$x=acb\cdots$
\qquad 
or
\qquad 
$x=\cdots dacb\cdots$
\item 
$y=cab\cdots$
\qquad 
or
\qquad 
$y=\cdots a'cab\cdots$ with $a'<a$\\
$x=acb\cdots$
\qquad 
or
\qquad 
$x=\cdots a'acb\cdots$
\end{enumerate}
\end{enumerate}
 
The proofs are similar for each case. 
We include a partial proof of case~\eqref{thm:knuth paths:non-KB:K1:bacd}. 
Suppose 
\begin{align*}
y&=\cdots bca\\
x&=\cdots bac
\end{align*}
or 
\begin{align*}
y &=\cdots bcad\cdots\\
x &=\cdots bacd\cdots
\end{align*}
where $a<b<c<d$. 
The idea is to show that $\localdecr_k(y)\leq \localdecr_k(x)$ and $\localdecr_k(x)\leq \localdecr_k(y)$ for all $k,$ from which the result follows. 

Let $k \geq 1$. 
To show $\localdecr_k(y)\leq \localdecr_k(x)$, suppose that $u_1,\dots,u_k$ are disjoint subsequences of $y$ such that 
$$
\localdecr_k(y) = \localdecr(u_1) + \cdots + \localdecr(u_k).
$$ 
We will produce disjoint subsequences $u_1',\dots,u_k'$ of $x$ where 
\[
\localdecr(u_1)+\cdots + \localdecr(u_k)\leq \localdecr(u_1')+\cdots + \localdecr(u_k').
\]
First, suppose that $c$ and $a$ are in different subsequences. 
Then set $u_i'\vcentcolon=u_i$ for each $1\leq i\leq k.$ 
Since 
$\localdecr(u_1)+\cdots + \localdecr(u_k) 
=
\localdecr(u_1')+\cdots + \localdecr(u_k')$,  
we have $\localdecr_k(y)\leq \localdecr_k(x)$.
    
Next, suppose that $b,c,$ and $a$ are in the same subsequence $u_j$ of $y$. 
Define $u_j'$ to be the subsequence of $x$ which is obtained from $u_j$ by swapping $c,a$ with $a,c$. Define $u_i'\vcentcolon= u_i$ for all~$i\neq j.$ 
Then,
since $a<b<c$, we have 
\[\localdecr(u_j) = \localdecr(\dots,b,c,a,\dots)\leq \localdecr(\dots,b,a,c,\dots)=\localdecr(u_j'),\]
so $\localdecr_k(y)\leq \localdecr_k(x).$
    
Lastly, suppose that $c$ and $a$ are in the same subsequence, say $u_1,$ and $b$ is in a different subsequence, say $u_2.$ 
Write $u_1$ as a concatenation 
\[u_1 = \underbrace{(\, \dots,c)}_{u_1^1} \sqcup \underbrace{(a,\dots \,)}_{u_1^2}\]
of two subsequences $u_1^1$ and $u_1^2,$ respectively. 
Write $u_2$ as a concatenation 
\[u_2 = \underbrace{(\, \dots,b)}_{u_2^1} \sqcup \underbrace{(\, \dots \,)}_{u_2^2} \]
of two subsequences $u_2^1$ and $u_2^2,$ respectively.
Define 
\begin{align*}
u_1'&\vcentcolon=u_2^1 \sqcup u_1^2 ={(\, \dots,b)} \sqcup {(a,\dots \,)},\\ 
u_2'&\vcentcolon= u_1^1 \sqcup u_2^2 = {(\, \dots,c)} \sqcup {(\, \dots \,)},
\end{align*}
and $u_i'\vcentcolon=u_i$ for all $i\notin \{ 1,2 \}$.  
Then, since $a<b<c$, 
\[\localdecr(u_1) + \localdecr(u_2)\leq \localdecr(u_1') + \localdecr(u_2'),\] 
so  $\localdecr_k(y)\leq \localdecr_k(x)$. 
The proof of the reverse inequality $\localdecr_k(x)\leq \localdecr_k(y)$ is similar. 
\end{proof}

Theorem~\ref{thm:knuth paths} allow us to use Knuth moves to find a subset of permutations whose soliton decomposition and RS insertion tableau coincide.

\begin{corollary}[Corollary of
Theorem~\ref{thm:tfae} and
Theorem~\ref{thm:knuth paths}]
\label{cor:Knuth moves preserve SD}
Let $w\in S_n$, let $T=\P{w}$, and let~$r$ be the row reading word of $T$.
\begin{enumerate}
\item
If $w$ is related to $r$ 
by a sequence of Knuth moves 
containing an odd number of $K_B$ moves,  
 then $\SD{w}\neq \P{w}=T$. 
\item 
If $w$ is related to $r$ 
by a sequence of 
non-$K_B$ moves, 
then $\SD{w}=\P{w}=T$. 
 \end{enumerate}
\end{corollary}

\begin{example}
The permutations $362514$ and $632514$ are
the reading words of the tableaux 
\[\ytableausetup{centertableaux, boxsize = 1.3em}
\quad\quad
\begin{ytableau}
 1 & 4 \\
 2 & 5\\
 3 & 6
 \end{ytableau} ~~ \text{ and } ~~
 \begin{ytableau}
 1 & 4 \\
 2 & 5\\
 3 \\
 6
 \end{ytableau} ~~ \text{ respectively}.
\]
Figure~\ref{fig:362514 with SD} (respectively, Figure~\ref{fig:knuth shape 2211})
 shows all permutations in the Knuth equivalence class \linebreak of~$r=362514$ (respectively, $r=632514$). 
The corresponding soliton decomposition is drawn next to each permutation. 
An edge with label $K_1$ (respectively, $K_2$) indicates that 
the Knuth move is a proper $K_1$  (respectively, $K_2$) move.
An edge with label $K_B$ indicates that the Knuth move is 
both $K_1$ and $K_2$. 
The permutations are arranged such that they form a subdiagram of the Hasse diagram of the right weak order\footnote{For definition of the right weak order, see, for example, the textbook~\cite[Section~3.1]{BB05}.} on the symmetric group $S_6$. 
\end{example}

\begin{figure}[p]
\centering 
\begin{tikzpicture}[yscale=1.1]
\node (top) at (0,0) {$\mathbf{r=362514}$, $t=0$};

\node at (2.3,0.1)
{$\small \young(14,25,36)$};

\node (left) at (-1.5,-1.5) {$362154$, $t=2$};
\node at (-3.4,-1.5) {$\small \young(14,25,6,3)$};
\node (right) at (1.6,-1.5) {$326514$, $t=1$};
\node at (3.4,-1.5) {$\small \young(14,25,6,3)$};
\node (bottom) at (0,-3) {$326154$, $t=2$};
\node at (-1.8,-3.4) {$\small \young(14,25,6,3)$};
      
\node (mostbottom) at (0,-4.5) {$321654$, $t=1$};
\node at (1.8,-4.4) {$\small\young(14,5,6,2,3)$};  
\draw [red, thick, shorten <=1pt, shorten >=1pt] (top) -- (left) node[pos=0.5, left]{$K_B$};
\draw [red, thick, shorten <=1pt, shorten >=1pt] (top) -- (right) node[pos=0.5, right]{$K_B$};
  
\draw [blue, thick, shorten <=1pt, shorten >=1pt] (bottom) -- (left) node[pos=0.5, left]{$K_1$};
\draw [blue, thick, shorten <=1pt, shorten >=1pt] (bottom) -- (right) node[pos=0.5, right]{$K_2$};
  
\draw [red, thick, shorten <=1pt, shorten >=1pt] (bottom) -- (mostbottom) node[pos=0.5, right]{$K_B$};
\end{tikzpicture}
\caption{The Knuth equivalence class of $r=362514$, with their soliton decompositions and steady-state times.}
\label{fig:362514 with SD}
\label{fig:362514 with ss}

\end{figure}
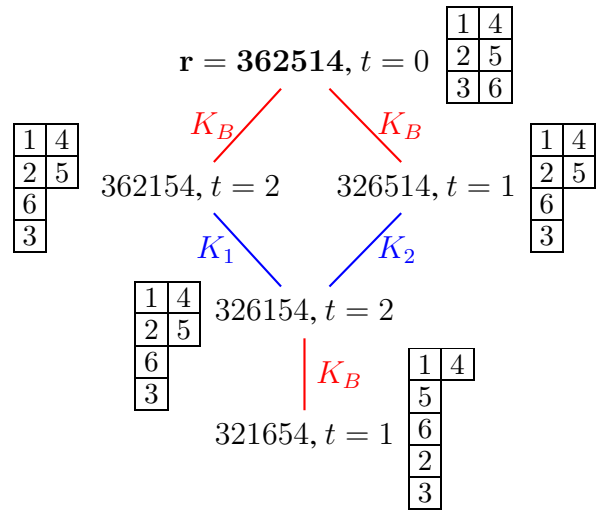
\begin{figure}

\begin{tikzpicture}[xscale=1.5,yscale=1.45]

\node (1) at (0,0) {$326541$, 
 $t=1$}; 
 \node at (-1.1,0) {\young(14,2,5,6,3)};

\node (2) at (0,2) {$362541$, $t=2$};
\node at (-1.3, 2) {\young(14,25,3,6)};

\node (4) at (2,4) {$365241$, $t=2$};
\node at (0.7, 4) {\young(14,25,3,6)};

\node (3) at (2,2) {$365214$, $t=1$};
\node at (2, 1.1) {\young(14,25,3,6)};

\node (8) at (4,4) {$635214$, $t=1$};
\node at (4, 2.9) {\young(14,25,3,6)};

\node (9) at (4,6) {$635241$, $t=2$};
\node at (2.7, 6) {\young(14,25,3,6)};

\node (7) at (6,4) {$632541$, $t=1$};
\node at (7.3, 4) {\young(14,2,5,3,6)};

\node (6) at (6,2) {$\mathbf{r=632514}$, $t=0$};
\node at (7.6, 2) {\young(14,25,3,6)};

\node (5) at (6,0) {$632154$, $t=1$};
\node at (7.3, 0) {\young(14,5,2,3,6)};

\draw [red, thick 
] (1) -- (2) node[pos=0.5, right]{$K_B$};

\draw [blue, thick] (2) -- (4) node[pos=0.4, right]{$K_2$};
\draw [blue, thick] (3) -- (4) node[pos=0.55, right]{$K_1$};
\draw [blue, thick] (3) -- (8) node[pos=0.4, right]{$K_2$};
\draw [blue, thick] (8) -- (9) node[pos=0.45, right]{$K_1$};
\draw [blue, thick] (4) -- (9) node[pos=0.4, right]{$K_2$};

\draw [blue, thick] (6) -- (8) node[pos=0.55, right]{$K_1$};

\draw [red, thick] (7) -- (9) node[pos=0.55, right]{$K_B$};

\draw [red, thick] (5) -- (6) node[pos=0.5, right]{$K_B$};

\end{tikzpicture}
\caption{The Knuth equivalence class of $r=632514$, with their soliton decompositions and steady-state times.}
\label{fig:knuth shape 2211}
\end{figure}
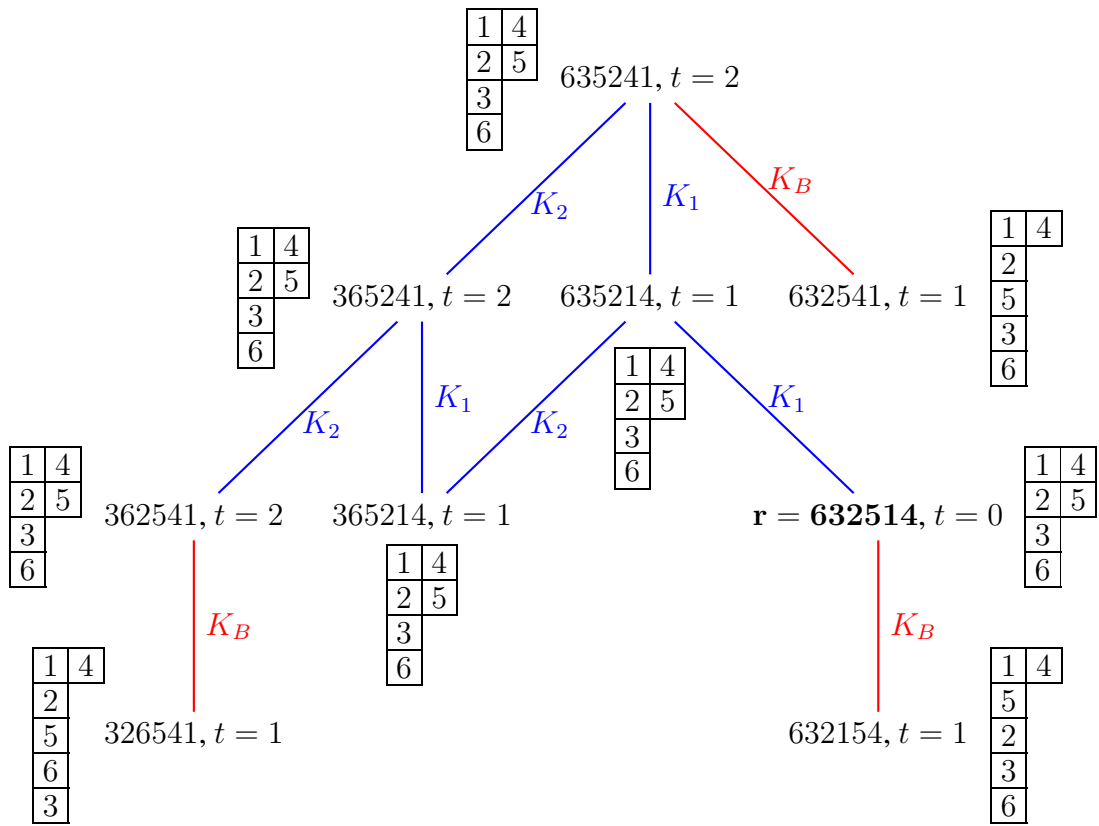

\subsection{Permutations one proper Knuth move from a row reading word}
\label{sec: t=1}

\begin{theorem}\label{thm:t=1}
Let $r$ be the row reading word of a standard tableau. 
Suppose $w$ is a permutation 
which is related to $r$ by 
one proper $K_1$ move or one proper $K_2$ move. 
Then $w$ has steady-state time~$1$.
\end{theorem}

If $w$ is one $K_B$ move from the row reading word 
of a standard tableau, then the steady-state time of $w$ may be $1$ or greater than $1$. 
See Example~\ref{ex:362514 with ss}. 
\begin{example}
\label{ex:362514 with ss}
The corresponding steady-state times are given next to each permutation in the two Knuth equivalence classes of Figures~\ref{fig:362514 with ss} and~\ref{fig:knuth shape 2211}.

In Figure~\ref{fig:362514 with ss}, 
the permutation $362154$ is one $K_B^-$ move from $r$, and its steady-state time is~$t=2$.  
 Another permutation, $326514$, is also 
 one $K_B^-$ move from $r$, and 
 its steady-state time is $t=1$.

In Figure~\ref{fig:knuth shape 2211}, we can perform a $K_B^-$ move and also a proper $K_1^+$ move on $r$ (see Lemma~\ref{lem: K_1^+ Config}).
The permutation $635214$ is one proper $K_1^+$ move from $r$, and its steady-state time is $t=1$, illustrating Lemma~\ref{lem: K_1^+ ss time 1}. 
\end{example}

\subsubsection{Proof of Theorem~\ref{thm:t=1}}
Theorem~\ref{thm:t=1} follows from the following four lemmas.

\begin{lemma}\label{lem: K_1^+ Config}
Let $r=r_1 r_2 \dots r_n$ be the row reading word of a standard tableau $P$.
\begin{enumerate}
\item 
\label{lem: K_1^+ Config:K_1^-}
If one performs a $K_1^-$ move on $r$, the move is $K_B$. 

\item 
\label{lem: K_1^+ Config:P}
Suppose we are able to perform a $K_1^+$ move
$ yxz \mapsto yzx $ 
(where $x < y < z$) 
on $r$. 
If~$r_1\neq y$, we must have 
\begin{equation}
\label{eq:lem: K_1^+ Config}
r= \underbrace{r_1 \, \dots \, r_\ell \, y \, \, x }_{\text{decreasing}} \, z \, \dots \, r_{n-1} \, r_n 
\end{equation}
where $r_1 > r_2 > \dots > r_\ell > y > x$.
The tableau $P$ must be of the form given in Figure~\ref{fig: configK1+}, where 
the entry $y$ is in its own row, and the row immediately above $y$ starts with entries~$x,z$. 
    
\item 
\label{lem: K_1^+ Config:cannot be KB}
If one performs a $K_1^+$ move on $r$, the move is \emph{not} $K_B.$    

\end{enumerate}    
\end{lemma}

\begin{figure}[ht]
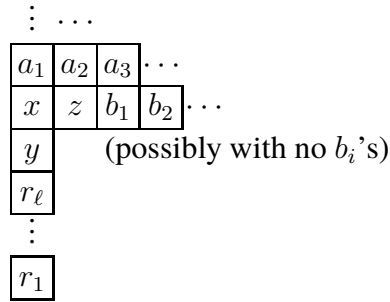

\centering
\[\begin{ytableau}
\none[\vdots] & \none[\cdots] \\
a_1 & a_2 & a_3 & \none[\cdots] \\
x & z & b_1 & b_2 & \none[\cdots] \\
y & \none & \none & \none & \none & \none[\text{(possibly with no $b_i$'s)}]\\
r_\ell\\
\none[\vdots]\\
r_1
\end{ytableau}
\]
\caption{General form of a standard tableau $P$ whose row reading word can undergo a $K_1^+$ move.}
\label{fig: configK1+}
\end{figure}

\begin{proof}
First, we prove part~\eqref{lem: K_1^+ Config:K_1^-} of the lemma. 
Suppose we perform a $K_1^-$ move  
$yzx \mapsto yxz$
(where $x<y<z$)
on $r$. 
Since $r$ is the row reading word of $P$, the tableau $P$ must contain a subtableau
\[
\ytableausetup{centertableaux, boxsize = 1.5em}
\begin{ytableau}
x & b \\
 y & z
\end{ytableau}
\qquad \text{ or } \qquad
\begin{ytableau}
x & \none & \none[\dots] & b\\
\none & \none[\dots] &y & z
\end{ytableau}.
\]
Since the rows and columns of $P$ are increasing, 
we must have
$x<b<z$. 
Thus, 
$r$ must contain a consecutive subsequence $yzxb'$ where $x<b' \leq b <z$, so
the $K_1^-$ move
$yzx \mapsto yxz$
is $K_B^-.$

Now suppose we perform a $K_1^+$ move 
$yxz \mapsto yzx$
on $r$. 
First, we prove part~\eqref{lem: K_1^+ Config:P}. 
\linebreak Since~$x<y<z$ and $P$ is standard, 
the entry
$y$ must be the only element in its row in $P$, that is, the rows of $P$ containing $x,y,z$ are of the form
\[\ytableausetup{centertableaux}
\begin{ytableau}
x & z&\none[\dots]\\
y
\end{ytableau}\]
If $r_1=y$, then we are done. 
Suppose $r_1 \neq y$,  and write $r=r_1 r_2 \dots r_\ell y x z \dots r_n$. 
Since the rows of $P$ are weakly decreasing in length, the rows of $P$ below $y$ are of size $1$. 
Since $P$ is standard, we have $r_1 > r_2 > \dots > r_\ell > y$. 
So $r$ is of the form given in~\eqref{eq:lem: K_1^+ Config} and $P$ is of the form given in Figure~\ref{fig: configK1+}.

Finally, 
to prove part~\eqref{lem: K_1^+ Config:cannot be KB} of the lemma, 
we prove that this $K_1^+$ move is not a $K_B$ move. 
If~$r_n = z$, then we know this $K_1^+$ move is not $K_B$. 
Suppose $r_n \neq z$, so $r=r_1 \dots yxzb \dots r_n$ for some $b$. 
Since $r$ is the row reading word of $P$,
either the entry $b$ is immediately above $x$ in $P$ or 
the entry $b$ 
is immediately to the right of $z$ in $P$:
$$
\ytableausetup{centertableaux}
\begin{ytableau}
b & \none[\cdots] \\
x & z \\
y
\end{ytableau}
\qquad\text{ or }\qquad
\ytableausetup{centertableaux}
\begin{ytableau}
x & z& b & \none[\cdots]\\
y
\end{ytableau}$$
Since $P$ is standard, either 
$b<x$ or
$z<b$. 
Either way, this $K_1^+$ move is not $K_B$. 
\end{proof}

\begin{lemma}\label{lem: K_2^- Config}
Let $r=r_1 r_2 \dots r_n$ be the row reading word of a standard tableau $P$.
\begin{enumerate}
\item 
\label{lem: K_2^- Config:K_2^+ is impossible}
It is impossible to perform a $K_2^+$ move on $r.$ 

\item 
\label{lem: K_2^- Config:P}
 
Suppose we are able to perform a $K_2^-$ move
$ zxy \mapsto xzy $
(where $x < y < z$)
which is not a $K_B$ move
on $r$. 
If $r_1 \neq z$, we have 
\begin{equation}
\label{eq:lem: K_2^- Config}
r= \underbrace{r_1 \dots r_\ell \, z \, x}_{\text{decreasing}} \, y \dots \, r_{n-1} \, r_n 
\end{equation}
where 
$r_1 > r_2 > \dots > r_\ell > z$. 
The tableau $P$ must be of the form given in Figure~\ref{fig: configK2+},
where 
the entry $z$ is in its own row, and the row immediately above $z$ starts with entries~$x,y$.
\end{enumerate}
\end{lemma}

\begin{figure}[ht]
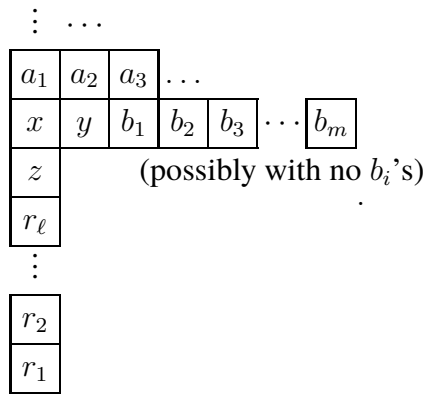

\centering
\[
\begin{ytableau}
\none[\vdots] 
 & \none[\cdots] \\
a_1 & a_2 & a_3 & \none[\dots]
\\
x & y & b_1 & b_2 & b_3 & \none[\cdots] & b_m \\
z & \none & \none & \none & \none & \none[\text{(possibly with no $b_i$'s)}] \\
r_\ell\\
\none[\vdots]\\
r_2\\
r_1
\end{ytableau}. 
\]
\caption{\emph{General form of a standard tableau $P$ whose row reading word can undergo a $K_2^-$ move which is not $K_B$}.}
\label{fig: configK2+}
\end{figure}

\begin{proof}
First, we prove part~\eqref{lem: K_2^- Config:K_2^+ is impossible} of the lemma. 
Assume (for the sake of contradiction) that one could perform a $K_2^+$ move on $r$. 
Then $r$ must contain a $xzy$ pattern.
Hence, since $r$ is the row reading word of $P$, the tableau $P$ must contain the following subtableau:
$$\ytableausetup{centertableaux}
\begin{ytableau}
 y &\none[\hdots]\\
\none & \none & x & z
\end{ytableau}$$
Notice that $y$ is north or northwest of $x$ but $x<y.$ 
This is a contradiction to the fact that $P$ is a standard tableau. 
Therefore, we cannot perform a $K_2^+$ move on $r$.

Next, we prove part~\eqref{lem: K_2^- Config:P} of the lemma.
Suppose we perform a $K_2^-$ move 
$zxy \mapsto xzy$ 
on $r$ which is not a $K_B$ move. 
If $r_1 = z$, then the last two rows of $P$ are of the form 
\[
\ytableausetup{centertableaux}
\begin{ytableau}
x & y & \none[\cdots]\\
z
\end{ytableau},
\]
so $P$ is of the form given in Figure~\ref{fig: configK2+}.

Suppose $r_1 \neq z$, 
and write $r=r_1 \, \dots \, r_\ell \, z \, x \, y \, \dots \, r_{n-1} \, r_n$. 
Since our $K_2^-$ move is not $K_B$, we must have either $r_\ell < x$ or $z < r_\ell$. 
Since $P$ is standard and $x$ is in the first column, we cannot have $r_\ell < x$. So $z < r_\ell$. 
Therefore $z$ is in its own row in $P$. 
Since the rows of $P$ are weakly decreasing in length, the rows of $P$ below $z$ are of size $1$. 
Since $P$ is standard, we have $r_1 > r_2 > \dots > r_\ell$. 
So $r$ is of the form given in~\eqref{eq:lem: K_2^- Config} and $P$ is of the form given in Figure~\ref{fig: configK2+}. 
\end{proof}

\begin{remark}
In general, a $K_2^-$ move on the row reading word of a standard tableau may (or may not) be $K_B$. 
\end{remark}

The proofs of the next two lemmas, Lemmas~\ref{lem: K_1^+ ss time 1} and~\ref{lem: K_2^- ss time 1}, are similar.

\begin{lemma}
\label{lem: K_1^+ ss time 1}
Suppose $r = r_1 r_2 \dots r_n \in S_n$ is the row reading word of a standard tableau $P$. 
Let $w$ be a permutation
which differs from $r$ by one proper $K_1$ move. 
Then $w$ first reaches its steady state at $t=1$. 
\end{lemma}
\begin{proof}
By Lemma~\ref{lem: K_1^+ Config}, 
applying a $K_1$ move that is not $K_B$ to $r$ must be a $K_1^+$ move
$yxz \mapsto yzx$ such that  
\begin{align*}
r &=\underbrace{r_1 \, r_2 \, \dots \, r_\ell \, y}_{\text{decreasing}} \, x \, z \, \dots \, r_{n-1} \, 
r_n 
\\
w=K_1^+(r) &=
\underbrace{r_1 \, r_2 \, \dots \, r_\ell \, y}_{\text{decreasing}}z 
\, x \, \dots \, r_{n-1}
\, r_n
\end{align*}
where 
$r_1 > r_2 > \dots > r_\ell > y$ (if $r_1 \neq y$) 
and $x< y <z$.

We apply the carrier algorithm to $w$.
First, we insert $r_1, r_2, \dots, r_\ell, y$ into the carrier.
Since these are decreasing, we eject $e, r_1, r_2, \dots, r_\ell$ from the carrier in consecutive order: 
\begin{align*}
\underbracket{\, e \, e \, \cdots \, e \, }_{\text{carrier}} 
&
\,
\underbrace{r_1 \, r_2 \, r_3 \, \dots \, r_{\ell-1} \, r_\ell \, y}_{\text{decreasing}} \, z \, x \, \dots \, r_n 
\\  
e 
\underbracket{\, r_1 \, e \, e \, \cdots \, e \, }
&
\underbrace{r_2 \, r_3 \, \dots \, r_{\ell-1} \, r_\ell \, y}_{\text{decreasing}} \, z \, x \, \dots \, r_n 
\\ 
e \, r_1 \, 
\underbracket{\, r_2 \, e \, e \, \cdots \, e \, }
& 
\, 
\underbrace{ r_3 \dots \, r_{\ell-1} \, r_\ell \, y}_{\text{decreasing}}\, z \, x \,  \dots \, r_n 
\\
& \vdots 
\\
e \, r_1 \, r_2 \, \dots \, r_\ell \, 
\underbracket{ \, y \, e \, e \, \cdots \, e \, }
&
 \, z \, x \, \dots \, r_n
\end{align*}
Next, we insert $z$ into the carrier.
Since the only non-$e$ entry in the carrier, $y$, is smaller than $z$, we eject an $e$: 
\begin{align*}
e \, e \, r_1 \, r_2 \dots r_\ell \, e 
\underbracket{ \, y \, z \, e \, e \, \cdots \, e \, }
&
\, x \, r_{\ell+4} \, \dots \, r_n
\end{align*}
Next, we insert $x$ into the carrier. 
Since $x <y < z$, we eject $y$ and get 
\begin{align*}
e \, e \, r_1 \, r_2 \dots r_\ell \, e \, y \, 
\underbracket{ \, x \, z \, e \, e \, \cdots \, e \, }
&
\, r_{\ell+4} \, \dots \, r_n
\end{align*}

Note that the string
\[
x \, z \, r_{\ell+4} \, \dots \, r_{n-1} \, r_n
\]
is equal to the consecutive subsequence $r_{\ell+2} \, \dots r_{n-1} \, r_n$ of $r$.  
This string 
is the row reading word of the subtableau (possibly with no $b_i$'s)
\[\begin{ytableau}
\none[\vdots] 
 & \none[\cdots] \\
 a_1 & a_2 & a_3 & \none[\cdots] \\
  x & z & b_1 & b_2 & \none[\cdots] 
\end{ytableau} 
\] 
of $P$, where $P$ is given in Figure~\ref{fig: configK1+}. 
Since this subtableau has the shape of a partition and 
 has increasing rows and columns, completing the carrier algorithm yields the configuration at time~$t=1$: 
\begin{align*}
e \, e \, r_1 \, r_2 \dots r_\ell \, e \, y \, 
\overbrace{e \, e \, \dots \, e}
^{\substack{
0 \text{ or more}\\ \text{copies}}}  
\,
x \, z \, 
 \, b_1 \, b_2 \, b_3 \,  \, \dots \, 
&
\, a_1 \, a_2 \, \dots \, \dots
r_{n-1} \, r_n \, 
\underbracket{\, e \, e \, \cdots \, e \, }.
\end{align*}

The configuration array at $t=1$ is the skew tableau created by taking $P$ and shifting some of the rows to the right. Since $P$ is standard tableau with partition shape to begin with, the configuration array is a standard skew tableau with weakly increasing rows. 
By Proposition~\ref{prop:t=0 generalization}, the configuration at $t=1$ is in steady state.
\end{proof}

\begin{lemma}
\label{lem: K_2^- ss time 1}
Suppose $r = r_1 r_2 \dots r_n \in S_n$ is the row reading word of a standard tableau $P$. 
Let $w$ be a permutation
which differs from $r$ by one proper $K_2$ move. 
Then $w$ first reaches its steady state at $t=1$.
\end{lemma}
\begin{proof}
By Lemma~\ref{lem: K_2^- Config}, 
applying a $K_2$ move that is not $K_B$ to $r$ must be 
a $K_2^-$ move 
$zxy \mapsto xzy$ to $r$  
such that 
\begin{align*}
r= r_1 \dots r_\ell \, & z \, x \, y 
\, \dots \, r_{n-1} \, r_n
\\
w = K_2^-(r) = r_1 \dots r_\ell \, & x \, z \, y \, \dots \, r_{n-1} \, r_n
\end{align*}
where
$r_1 > r_2 > \dots > r_\ell > z$ (if $r_1 \neq z$) and $x<y<z$.

As in the proof of Lemma~\ref{lem: K_1^+ ss time 1}, 
we apply the carrier algorithm to $w$.
We insert the decreasing sequence $r_1, r_2, \dots, r_\ell, x$ into the carrier and eject $e, r_1, r_2, \dots, r_\ell$, in that order. 
As we insert $z$ and $y$, we eject $e$ and $z$, in that order: 
\begin{align*} 
\underbracket{\, e \, \cdots \, e \, }_{\text{carrier}} 
&
\, \underbrace{r_1 \, r_2 \, r_3 \, \dots r_\ell \, x}_{\text{decreasing}} \, z \, y \, \dots \, r_{n-1} \, r_n 
\\ 
e \, 
\underbracket{\, r_1 \, e \, e \, \cdots \, e \, } 
&
\, \underbrace{r_2 \, r_3 \, \dots r_\ell \, x}_{\text{decreasing}} \, z \, y  \, \dots \, r_{n-1} \, r_n 
\\ 
e \, r_1 \,
\underbracket{\, r_2 \, e \, e \, \cdots \, e \, } 
&
\, \underbrace{r_3 \dots r_\ell \, x}_{\text{decreasing}} \, z \, y \, \dots \, r_{n-1} \, r_n 
\\
& \vdots 
\\
e \, r_1 \, r_2 \, \dots \, r_\ell \, 
\underbracket{\, x \, e \, 
\dots \, e \, } 
&
 \, z \, y \, \dots \, r_{n-1} \, r_n 
\\
e \, r_1 \, r_2 \, \dots \, r_\ell \, e \, 
\underbracket{\, x \, z \, e \, 
\dots \, e \, } 
&
  \, y \, r_{\ell+4} \, \dots \, r_n
\\
e \, r_1 \, r_2 \, \dots \, r_\ell \, e \, z \,
\underbracket{\, x \, y \, e \, 
\dots \, e \, } 
&
  \, r_{\ell+4}\, \dots r_n
\end{align*}
Note that the string
\[
x \, y \, r_{\ell+4} \, \dots \, r_{n-1} \, r_n
\]
is equal to the consecutive subsequence $r_{\ell+2} \, \dots r_{n-1} \, r_n$ of $r$.  This string 
is the row reading word of the subtableau (possibly with no $b_i$'s)
\[\begin{ytableau}
\none[\vdots] 
 & \none[\cdots] \\
 a_1 & a_2 & a_3 & \none[\cdots] \\
  x & y & b_1 & b_2 & \none[\cdots]
\end{ytableau} 
\]
of $P$, where $P$ is given in Figure~\ref{fig: configK2+}. 
Since this subtableau has the shape of a partition and 
 has increasing rows and columns, completing the carrier algorithm yields the configuration at time~$t=1$: 
\begin{align*}
e \, e \, r_1 \, r_2 \dots r_\ell \, e \, z \, 
\overbrace{e \, \dots \, e}^{\substack{0 \text{ or more} \\ \text{copies}}} 
\,
x \, y \, 
 \, b_1 \, b_2 \, b_3 \, \cdots \, 
&
\, 
\, a_1 \, a_2  \, \dots
\, 
r_{n-1} \, r_n \, 
\underbracket{\, e \, e \, \cdots \, e \, }.
\end{align*}

The configuration array at $t=1$ is the skew tableau created by taking $P$ and shifting some of the rows to the right.
Since $P$ is standard tableau with partition shape to begin with, the configuration array is a standard skew tableau with weakly increasing rows.
By Proposition~\ref{prop:t=0 generalization}, the configuration at $t=1$ is in steady state.
\end{proof}



\section*{Acknowledgements}
This research project started during the University of Connecticut 2020 Mathematics REU. 
Our project was inspired by a blog post~\cite{Lewis20Blog} for the University of Minnesota's Open Problems in Algebraic Combinatorics (OPAC) conference and conversations with Joel B. Lewis. 
We thank Ian Whitehead for serving as a faculty mentor to B.~Drucker's research course in Fall 2020 and for helpful suggestions. 
We also thank Pavlo Pylyavskyy and Rei Inoue for useful comments and Marisa Cofie, Olivia Fugikawa, Madelyn Stewart, and David Zeng for many discussions during SUMRY 2021. 
Special thanks to Darij Grinberg for proving one of our conjectures and for helpful feedback. 
This work also benefited from computation using {\sc SageMath}~\cite{sage} and the High Performance Computing 
facility at the University of Connecticut. 
Finally, we thank the anonymous referees whose suggestions helped improve and clarify this paper.

\bibliographystyle{alphaurl}
\bibliography{dggrs}

\end{document}